\documentclass[11pt,reqno]{amsart}

\usepackage[T1]{fontenc}
\usepackage[utf8]{inputenc}
\usepackage{lmodern}
\usepackage{microtype}
\usepackage[a4paper,margin=2.5cm]{geometry}
\usepackage{amsmath,amssymb,mathtools}
\usepackage{mathrsfs}
\usepackage{booktabs}
\usepackage{float}
\usepackage{graphicx}
\usepackage{array}
\usepackage{tabularx}
\usepackage{enumitem}
\usepackage{xcolor}
\usepackage{hyperref}
\usepackage{cleveref}

\hypersetup{
  colorlinks=true,
  linkcolor=blue!45!black,
  citecolor=blue!45!black,
  urlcolor=blue!55!black,
  pdftitle={Morse Topology, Zero Products, and Elliptic Prime Crossings},
  pdfauthor={Michel Planat}
}

\setlist[itemize]{leftmargin=1.6em,itemsep=0.28em,topsep=0.4em}
\setlist[enumerate]{leftmargin=1.8em,itemsep=0.38em,topsep=0.45em}
\allowdisplaybreaks
\newtheorem{theorem}{Theorem}
\newtheorem{proposition}{Proposition}
\newtheorem{conjecture}{Conjecture}
\newtheorem{problem}{Problem}
\theoremstyle{remark}
\newtheorem{remark}{Remark}
\newcommand{\Q}{\mathbb{Q}}
\newcommand{\R}{\mathbb{R}}
\newcommand{\T}{\mathbb{T}}
\newcommand{\Fq}{\mathbb{F}}
\newcommand{\ShaE}{\mathop{\mathrm{Sha}}}
\newcommand{\Reg}{\operatorname{Reg}}
\newcommand{\ord}{\operatorname{ord}}

\newcommand{\Corr}{\operatorname{Corr}}

\newcommand{\ST}{\mathrm{ST}}
\newcommand{\e}{\mathrm{e}}

\newcolumntype{P}[1]{>{\raggedright\arraybackslash}p{#1}}

\newcommand{\pacsline}[1]{%
  \begingroup
  \par\vspace{-0.6em}
  \begin{center}\small\textsc{pacs numbers:}\ #1\end{center}
  \par\vspace{0.2em}
  \endgroup}

\title[Morse Topology and Elliptic Prime Crossings]
{Morse Topology, Zero Products,\\ and Elliptic Prime Crossings\\[0.45em]
\normalfont\large Arithmetic Information at the Interface of ECRH and BSD}

\author{Michel Planat}
\address{Universit\'e Marie et Louis Pasteur, Institut FEMTO-ST, CNRS,
Besan\c{c}on, France}
\email{michel.planat@femto-st.fr}

\date{August 2026}

\subjclass[2020]{Primary 11G40, 11M26;
Secondary 11G05, 11N05, 37A44, 57R70, 60B20}

\keywords{Elliptic curve $L$-functions; Birch--Swinnerton-Dyer conjecture;
Hadamard product; low-lying zeros; odd orthogonal symmetry; hard edge;
Morse theory; Kronecker flow; shrinking targets; Chebyshev bias}

\begin{document}

\begin{abstract}
We study how the noncentral zeros of an elliptic-curve $L$-function enter the Birch--Swinnerton-Dyer (BSD) leading coefficient and the prime-scale crossings of an elliptic Chebyshev race.  Throughout, ECRH abbreviates the \emph{Riemann hypothesis for the Hasse--Weil $L$-function of a single elliptic curve}: in the normalization used here, every nontrivial zero of $L(E,s)$ lies on the line $\Re s=1$.  No relation to the classical Riemann hypothesis is assumed or asserted.  The completed $L$-function gives an exact Hadamard decomposition, and under ECRH the two zero statistics
\[
  J_E=\sum_{\gamma_E>0}\log\!\left(1+\frac{1}{4\gamma_E^2}\right),
  \qquad
  V_E=2\sum_{\gamma_E>0}\frac{1}{\tfrac14+\gamma_E^2}
\]
satisfy the exact positive expansion $J_E=V_E/8+\Delta_E$.  For seventeen rank-one isogeny classes of conductor $50700$, the first zero contributes on average $82.4\%$ of $\Delta_{E,20}$, while independent edge values reveal a discrete $35/36$ zero-count split in the tail beyond height $20$.

For Haar $SO(2N+1)$ we prove the fixed-dimensional hard-edge law $\Pr\{\mathcal J_N>y\}\sim C_Ne^{-3y/2}$ and explain why finite-conductor comparisons require an explicit convention for nonintegral effective dimension.  Every finite zero truncation also defines a Morse function on a phase torus; its critical values, level sets and coarea density provide a rigorous topological model for crossing statistics.

The final construction is modelled on \emph{false Chebyshev primes}.  The term is not standard: it was introduced in earlier work \cite{PlanatSole2013,PlanatFalseCheb2026}, in the setting of the classical Riemann hypothesis, for primes whose own jump of the Chebyshev function carries a cumulative prime race across a reference level.  We transfer only that local mechanism and define elliptic crossing primes through the midpoint and half-jump of a logarithmically weighted Frobenius race.  A four-stage conductor-$50700$ computation shows agreement of the coarse arithmetic and zero-derived fields, but the omitted-zero noise above height $20$ exceeds the individual Frobenius windows.  The experiment therefore identifies a quantitative resolution barrier rather than a failure of the prime-sampled shrinking-target transfer.
\end{abstract}

\maketitle

\pacsline{02.10.De; 02.40.Sf; 02.40.Vh; 02.50.Cw; 05.45.-a}

\section{Introduction}
\label{sec:introduction}

The Hasse--Weil $L$-function of an elliptic curve carries two complementary kinds of arithmetic information.  Its central Taylor coefficient is governed conjecturally by Birch--Swinnerton-Dyer (BSD), while its noncentral zeros control explicit-formula fluctuations of prime sums.  This paper develops a common framework for these two aspects.  The first objective is to measure exactly how the noncentral zeros enter the central leading coefficient and how much arithmetic information remains after only finitely many zeros are known.  The second is to understand the low-zero hard edge through finite-dimensional orthogonal models.  The third is to interpret the zero-generated explicit formula as a phase flow whose level sets organize elliptic analogues of false Chebyshev primes.

Let $E/\Q$ have conductor $N_E$, Hasse--Weil function $L(E,s)$, and completed function
\begin{equation}
  \Lambda_E(s)=N_E^{s/2}(2\pi)^{-s}\Gamma(s)L(E,s).
  \label{eq:completed}
\end{equation}
It satisfies
\begin{equation}
  \Lambda_E(s)=w_E\Lambda_E(2-s),
  \qquad w_E\in\{\pm1\}.
  \label{eq:functional-equation}
\end{equation}
Write
\begin{equation}
  r=r_E=\ord_{s=1}L(E,s),
\end{equation}
so that $w_E=(-1)^r$.  BSD predicts that $r$ is the Mordell--Weil rank and identifies $L^{(r)}(E,1)/r!$ with periods, the regulator, Tamagawa factors, torsion and $\ShaE(E)$ \cite{BSD1965,Cassels1965}.

No implication between the classical Riemann hypothesis and ECRH is asserted.  The question is internal to the elliptic $L$-function: how do its noncentral zeros, its central leading coefficient and its prime-side Frobenius data constrain one another?

\paragraph{Standing hypotheses.}
Two hypotheses recur, and both are named wherever they are used.
\begin{itemize}
  \item \textbf{ECRH}, the Riemann hypothesis for $L(E,s)$: every nontrivial zero satisfies $\Re s=1$, so that $\rho_E=1\pm i\gamma_E$ with $\gamma_E$ real.  Only the zeros of the single curve $E$ are involved.
  \item \textbf{LI}, linear independence: the positive ordinates $\gamma_E$ are linearly independent over $\Q$.  This is the elliptic counterpart of the hypothesis under which Rubinstein and Sarnak obtained limiting logarithmic distributions for classical prime races \cite{RubinsteinSarnak1994}.  It is invoked only for limiting distributions, Bessel products and unique ergodicity of a phase flow; \cref{prop:torus-local-time} is the one place where it enters a proof.
\end{itemize}
Modularity of $E/\Q$ \cite{Wiles1995,BCDT2001} is used without further comment: it is what makes $\Lambda_E$ entire of order one, so that \eqref{eq:functional-equation} holds globally.  Statements requiring neither hypothesis are marked \emph{exact} in \cref{tab:epistemic-dictionary}.

\subsection{Main results}
The principal achievements are as follows.
\begin{enumerate}
  \item The Hadamard product gives an exact decomposition of the central leading coefficient.  Under ECRH the two zero statistics satisfy
  \[
    J_E=\frac18V_E+\Delta_E,
    \qquad
    \Delta_E=\sum_{k\ge2}\frac1k\sum_{\gamma_E>0}(1+4\gamma_E^2)^{-k}>0.
  \]
  Thus the observed $J_E$--$V_E$ relation is governed by an exact nonlinear remainder rather than by a universal regression law.
  \item For seventeen rank-one isogeny classes of conductor $50700$, the first zero accounts on average for $82.4\%$ of $\Delta_{E,20}$.  Independent evaluations of $L(E,3/2)$ determine the omitted zero-product tail and separate it into groups with $35$ and $36$ listed ordinates.
  \item For Haar $SO(2N+1)$ we prove a fixed-dimensional large-product tail with exponent $3/2$.  The exponent is stable under elliptic unfolding, whereas the finite-dimensional constant and the first-eigenangle law depend on the chosen conductor-to-matrix matching convention.
  \item A finite zero truncation is a Morse observable on $\T^M$.  Its signed amplitude sums are the critical values, and the coarea formula identifies the continuous level density.  This supplies a literal topological model for the level profiles used in the crossing construction.
  \item The midpoint mechanism of false Chebyshev primes \cite{PlanatFalseCheb2026} leads to intrinsic elliptic crossing primes.  The conductor-$50700$ computation shows that a height-$20$ zero truncation reproduces the coarse race but cannot resolve individual prime windows because the omitted-zero variance is too large.
\end{enumerate}

\subsection{Six layers, three of them topological}
\label{subsec:layers}

Programmes of this kind fail in a characteristic way: a statement established at one layer is quietly reused as though it had been established at another.  The layers are therefore fixed once here, and every later section is a statement about a named layer rather than about the framework as a whole.  \Cref{tab:epistemic-dictionary} records them, split into the three layers that carry genuine topological content and the four that do not.

\paragraph{The topological block.}
Exactly three layers are topological, and they are topological in three \emph{different} senses.  They are developed in \cref{sec:topological-interpretation} and collected in the synthesis \eqref{eq:topological-synthesis}.

\begin{enumerate}[label=\textbf{(T\arabic*)},leftmargin=2.9em]
  \item \emph{Cohomological, and local.}  For a good prime $q$ the Frobenius trace is literally $a_q(E)=\operatorname{tr}(\operatorname{Frob}_q\mid H^1_{\mathrm{\acute et}})$ on a fixed two-dimensional $\Q_\ell$-space, so the prime side of the race consists of cohomological invariants of single reductions.  This is the only place where an actual cohomology group appears, and it is local, finite-dimensional and fixed.  At no point are the global ordinates $\{\gamma_{E,k}\}$ identified with eigenvalues of an operator.
  \item \emph{Monodromy, and only conjecturally so over $\Q$.}  In function-field families the Zariski closure of a geometric monodromy group governs the distribution of Frobenius conjugacy classes, and Katz--Sarnak theory converts this into a compact classical group.  Over $\Q$ the group $SO(2N+1)$ is a symmetry \emph{model} for the low-lying zeros of a family, not a theorem about a single curve: the hard-edge tail law of \cref{sec:hard-edge-tail} is a consequence of Weyl measure, whereas the identification of an arithmetic family with that measure is a prediction.
  \item \emph{Differential-topological, and exact at each truncation.}  Every finite truncation of the explicit formula is a Morse function $\Phi_{E,M}$ on a phase torus $\T^M$, with $2^M$ critical points and critical values $\sum_k\varepsilon_kA_{E,k}$; its regular level sets are codimension-one submanifolds, and the coarea formula converts Kronecker occupation into a level density.  This layer is unconditional at fixed $M$, and it is the only one whose topology is used quantitatively.
\end{enumerate}

\paragraph{The remaining layers.}
The other four are analytic, measure-theoretic, conditional and open respectively, and none of them is a topological statement.  The Hadamard decomposition and the positive expansion $J_E=V_E/8+\Delta_E$ are exact identities about a single $L$-function.  Derivative excision is the layer most easily mistaken for topology: it deforms a probability measure along the fibres of $\log D_A$, and \cref{prop:conditional-invariance} shows that the conditional law is unchanged, so the deformation is measure-theoretic and not a deformation of the group.  It is treated inside \cref{sec:topological-interpretation} precisely in order to be excluded from \eqref{eq:topological-synthesis}.  The BSD information interval is conditional both on BSD and on a certified tail bound.  The prime-sampled shrinking-target transfer is open, and \cref{subsec:50700-crossing-pilot} measures how far from decidable it currently is.

\begin{table}[H]
\centering
\small
\caption{Mathematical layers and their logical status.  The upper block contains the three layers with genuine topological content, in the three distinct senses collected in \eqref{eq:topological-synthesis}; the lower block contains the analytic, measure-theoretic, conditional and open layers.  The last column points to the section in which each layer is developed.}
\label{tab:epistemic-dictionary}
\begin{tabularx}{\textwidth}{@{}P{2.45cm}P{3.05cm}P{2.35cm}>{\raggedright\arraybackslash}X P{1.35cm}@{}}
\toprule
Layer & Object & Status & Role & \S\\
\midrule
\multicolumn{5}{@{}l@{}}{\emph{Topological layers, in three distinct senses; synthesis in \eqref{eq:topological-synthesis}}}\\
\addlinespace[3pt]
\textbf{(T1)} Local arithmetic geometry
  & Frobenius on $H^1_{\mathrm{\acute et}}$
  & Exact
  & Produces $a_q(E)$ and the local Euler factors, hence the whole prime side of the race and of the crossing sets.
  & \ref{subsec:local-cohomology}, \ref{sec:race}\\
\addlinespace[3pt]
\textbf{(T2)} Family symmetry and monodromy
  & Haar $SO(2N+1)$ eigenangles; geometric monodromy in function-field analogues
  & Theorem in function-field families; statistical model over $\Q$
  & Models odd-orthogonal hard-edge repulsion and finite-conductor low-zero statistics; source of the $e^{-3y/2}$ tail law.
  & \ref{subsec:family-symmetry}, \ref{sec:hard-edge-tail}\\
\addlinespace[3pt]
\textbf{(T3)} Finite phase topology
  & Morse observable $\Phi_{E,M}:\T^M\to\R$, level sets and coarea density
  & Exact at fixed $M$; local-time limit under LI
  & Critical values are the signed amplitude sums; regular level sets organize crossings and continuous logarithmic-time occupation.
  & \ref{subsec:torus-flow}, \ref{subsec:crossing-topology}\\
\midrule
\multicolumn{5}{@{}l@{}}{\emph{Analytic, measure-theoretic, conditional and open layers}}\\
\addlinespace[3pt]
Global analytic identity
  & Hadamard decomposition and $J_E=V_E/8+\Delta_E$
  & Exact; zero-sum form under ECRH
  & Connects the central leading coefficient with the noncentral zero multiset; the positive $\Delta_E$ replaces any putative regression law.
  & \ref{sec:hadamard}, \ref{sec:endpoint}\\
\addlinespace[3pt]
Derivative excision \emph{(not topological)}
  & Reweighting $d\mu_W\propto W(\log D_A)\,d\mu_0$
  & Deformation of measure; conditional law invariant
  & Deliberately excluded from the block above: only the derivative marginal can be tested, since $x_1\mid\log D_A$ is unchanged.
  & \ref{subsec:excision-measure}, \ref{sec:dhkms-tests}\\
\addlinespace[3pt]
BSD information interval
  & Truncated zero product plus a certified tail bound
  & Conditional on BSD and on the bound
  & Gives consistency intervals for $|\ShaE(E)|$ and zero-list completeness diagnostics.
  & \ref{sec:bsd-information}\\
\addlinespace[3pt]
Arithmetic sampling
  & Prime-sampled shrinking targets around $\Sigma_{E,M}(c)$
  & Open
  & Transfers the phase-torus level geometry to actual Frobenius jumps; at present the resolution barrier is quantified, not the transfer.
  & \ref{sec:crossings}, \ref{subsec:50700-crossing-pilot}\\
\bottomrule
\end{tabularx}
\end{table}

\subsection{Organization and reading paths}
\Cref{sec:race} introduces the logarithmically weighted elliptic race and the origin of its mean.  \Cref{sec:topological-interpretation} develops \textbf{(T1)}--\textbf{(T3)} and separates them from the measure-theoretic layer.  \Cref{sec:hadamard,sec:endpoint} establish the exact zero-product identities and the positive nonlinear remainder.  \Cref{sec:numerics} analyzes the conductor-$50700$ data, the low-zero hard edge and the finite-dimensional orthogonal comparison.  \Cref{sec:bsd-information} develops the truncated BSD information interval and its zero-counting requirement.  \Cref{sec:crossings} defines the elliptic analogues of false Chebyshev primes and reports the finite-$X$ power analysis.  \Cref{sec:outlook} summarizes the remaining mathematical questions.

Three routes are self-contained.  A reader interested in the exact statements may read \cref{sec:hadamard,sec:endpoint} and \cref{sec:bsd-information}; a reader interested in the topological content may read \cref{sec:topological-interpretation} together with \cref{subsec:crossing-topology}; a reader interested only in the conductor-$50700$ evidence may read \cref{sec:numerics} and \cref{subsec:50700-crossing-pilot}.

\section{The elliptic race and the origin of its mean}
\label{sec:race}

\paragraph{Literature.}
Biased prime races go back to Chebyshev, and their modern conditional form is due to Rubinstein and Sarnak \cite{RubinsteinSarnak1994}: under the relevant Riemann hypothesis and LI, the logarithmically normalized error term has a limiting distribution whose characteristic function is an explicit Bessel product.  The elliptic case was raised by Sarnak in his letter to Mazur on the bias of $\tau(p)$ \cite{Sarnak2007} and set in a wider context by Mazur \cite{Mazur2008}; the conditional analysis for elliptic curves, and in particular the dependence of the bias on the rank, is due to Fiorilli \cite{Fiorilli2014}.  A general framework for limiting distributions of error terms of this shape is given by Akbary, Ng and Shahabi \cite{ANS2014}.  The explicit formula and the zero-counting estimates used throughout are standard \cite{IwaniecKowalski2004}.  The Hasse bound $|a_q(E)|\le2\sqrt q$ is a special case of the Weil conjectures proved by Deligne \cite{Deligne1974}; the Sato--Tate law for the angles $\theta_q$ of a non-CM curve is a theorem \cite{BLGHT2011}, and the symmetric-power automorphy needed to control the prime-power terms is supplied by Newton and Thorne \cite{NewtonThorne2021I,NewtonThorne2021II}.

For a good prime $q\nmid N_E$, set
\begin{equation}
  a_q(E)=q+1-\#E(\Fq_q),
  \qquad |a_q(E)|\le 2\sqrt q.
\end{equation}
The crossing problem is most cleanly formulated with the logarithmically weighted elliptic race
\begin{equation}
  \mathcal E_E^{*}(x)
  =-\frac{1}{\sqrt x}
   \sum_{\substack{q\le x\\q\nmid N_E}}
   \frac{a_q(E)}{\sqrt q}\log q.
  \label{eq:race}
\end{equation}
Its jump at a good prime is
\begin{equation}
  \mathcal E_E^{*}(q)-\mathcal E_E^{*}(q^-)
  =-\frac{a_q(E)\log q}{q},
  \label{eq:race-jump}
\end{equation}
The logarithmic weight avoids the systematic $2\mu_E/\log x+O(1/\log^2x)$ drift produced when an unweighted prime sum is converted by partial summation.

Under ECRH and the usual hypotheses controlling multiplicities and prime powers, the explicit formula has the schematic limiting form
\begin{equation}
  \mathcal E_E^{*}(\e^u)=\mu_E+Z_E(u)+o(1),
  \qquad \mu_E=2r-1.
  \label{eq:race-explicit}
\end{equation}
The term $2r$ comes from the central zero.  The additional $-1$ is the stable good-prime-square contribution.  Indeed, with $\lambda_p=a_p(E)/\sqrt p$, the normalized coefficient at $p^2$ is $\lambda_p^2-2$, whose Sato--Tate mean is $-1$.  Thus
\begin{equation}
  \underbrace{2r}_{\text{central zero}}
  \quad+\quad
  \underbrace{(-1)}_{\text{prime squares / }\operatorname{Sym}^2 E}
  \label{eq:mean-provenance}
\end{equation}
is the correct provenance of the mean.  For finite $x$, the centred prime-square fluctuation
\begin{equation}
  Q_{E,2}(x)
  =1+\frac1{\sqrt x}
   \sum_{\substack{p^2\le x\\p\nmid N_E}}
   (\lambda_p^2-2)\log p
  \label{eq:prime-square-fluctuation}
\end{equation}
and the smaller higher-prime-power and bad-prime terms should be retained whenever the target window is of order $\log q/\sqrt q$.  For a non-CM curve the symmetric-power analytic structure needed for the limiting mean is available through automorphy \cite{NewtonThorne2021II}.

Under ECRH the noncentral zeros have the form
\begin{equation}
  \rho_E=1\pm i\gamma_E,
  \qquad \gamma_E>0,
\end{equation}
and generate $Z_E$.  Under a linear-independence hypothesis, its centred limiting characteristic function is formally
\begin{equation}
  \widehat g_E(t)
  =\prod_{\gamma_E>0}
  J_0\!\left(\frac{2t}{\sqrt{\tfrac14+\gamma_E^2}}\right).
  \label{eq:bessel-product}
\end{equation}
The variance is
\begin{equation}
  V_E=2\sum_{\gamma_E>0}\frac{1}{\tfrac14+\gamma_E^2}.
  \label{eq:variance}
\end{equation}
The low zeros must be compared with the orthogonal symmetry type and finite-conductor repulsion appropriate to the family \cite{ILS2000,Miller2004,Miller2006,Marshall2013}.

\section{A layered topological interpretation}
\label{sec:topological-interpretation}

The word \emph{topological} enters this framework in several mathematically distinct senses.  They are complementary, but they must not be conflated.  The local Euler factors are cohomological; the low-zero statistics of a family are organized by a compact symmetry type; derivative excision is a deformation of probability measure rather than of the underlying group; and a finite truncation of the explicit formula defines a genuine quasiperiodic flow on a torus.  The section records the precise content of each layer and the limits of the interpretation.  The three layers that survive as topological are rows \textbf{(T1)}--\textbf{(T3)} of \cref{tab:epistemic-dictionary}; \cref{subsec:excision-measure} is included here in order to be excluded from them.

\subsection{Local cohomology and Frobenius traces}
\label{subsec:local-cohomology}

\paragraph{Literature.}
The $\ell$-adic cohomology of curves over finite fields, and the reading of $a_q(E)$ as a Frobenius trace, are treated in Deligne \cite{Deligne1974} and Milne \cite{Milne1980}; for elliptic curves specifically see Silverman \cite[Ch.~V]{Silverman2009}.  Explicit computation of the $a_q(E)$ and of the invariants used in \cref{sec:bsd-information} follows Cremona \cite{Cremona1997}.

For every good prime $q\nmid N_E$, let
\begin{equation}
  V_{\ell,q}(E)
  =H^1_{\mathrm{\acute et}}
  (E_{\overline{\Fq}_q},\Q_\ell),
  \qquad \ell\ne q.
\end{equation}
This is a two-dimensional $\Q_\ell$-vector space.  Geometric Frobenius has characteristic polynomial
\begin{equation}
  \det\!\left(T-\operatorname{Frob}_q\mid V_{\ell,q}(E)\right)
  =T^2-a_q(E)T+q,
  \label{eq:frobenius-characteristic-polynomial}
\end{equation}
so that
\begin{equation}
  a_q(E)
  =\operatorname{tr}\!\left(
  \operatorname{Frob}_q\mid V_{\ell,q}(E)
  \right),
  \qquad
  \det\!\left(
  \operatorname{Frob}_q\mid V_{\ell,q}(E)
  \right)=q.
  \label{eq:frobenius-trace}
\end{equation}
Thus the prime-side coefficients in \eqref{eq:race} are literally cohomological, and the local Euler factor is
\begin{equation}
  L_q(E,s)^{-1}
  =\det\!\left(
  I-q^{-s}\operatorname{Frob}_q
  \mid V_{\ell,q}(E)
  \right).
  \label{eq:local-factor-cohomology}
\end{equation}

This local statement must be distinguished from a global spectral assertion.  The infinitely many zeros $1\pm i\gamma_{E,k}$ are not known to be eigenvalues of Frobenius on the fixed two-dimensional space $H^1_{\mathrm{\acute et}}(E_{\overline\Q},\Q_\ell)$.  No cohomological Hilbert--P\'olya operator over $\Q$ is assumed here.  In particular,
\begin{equation}
  N_E(T)=\#\{k:0<\gamma_{E,k}\le T\}
\end{equation}
is a spectral counting function, not a Betti number or a varying cohomological dimension.  Its staircase structure reflects the discreteness of the zero multiset.

\subsection{Family symmetry and geometric monodromy}
\label{subsec:family-symmetry}

\paragraph{Literature.}
The monodromy route from function-field families to compact classical groups is developed in Katz and Sarnak \cite{KatzSarnak1999}, with the number-field philosophy summarized in \cite{KatzSarnak1999BAMS}.  One-level densities for families over $\Q$ were computed by Iwaniec, Luo and Sarnak \cite{ILS2000} and, for elliptic curves, by Miller \cite{Miller2004,Miller2006}; the excess repulsion near the central point observed at finite conductor was analysed by Marshall \cite{Marshall2013}.  Lower-order terms that make a finite-conductor comparison quantitative are given by Huynh, Keating and Snaith \cite{HKS2009}, and in the ratios framework by Conrey and Snaith \cite{ConreySnaith2007}.

The second layer is a family-level symmetry principle.  In suitable function-field families, a sheaf of cohomology groups carries a geometric monodromy representation, and the Zariski closure of its image controls the distribution of Frobenius conjugacy classes.  Katz--Sarnak theory then relates the limiting statistics to compact classical groups \cite{KatzSarnak1999}.  In that setting the route from monodromy to random matrices is genuinely geometric.

For elliptic curves over $\Q$, the compact group should be interpreted more cautiously.  It is a symmetry model for the low-lying zeros of a natural family, not a proven monodromy group attached to a single curve.  In an odd-sign family the forced central zero corresponds, on the random-matrix side, to the fixed eigenvalue $1$ of $SO(2N+1)$; the remaining positive eigenangles model the noncentral zeros.  The hard-edge repulsion and the tail law in \cref{sec:hard-edge-tail} are therefore consequences of the Weyl measure on the compact homogeneous space, while the identification of a given number-field family with that model remains an arithmetic-statistical prediction.

The effective matrix size
\begin{equation}
  N_{\rm std}=\log\frac{\sqrt{N_E}}{2\pi}
\end{equation}
is likewise a density-matching parameter.  It is not the dimension of $H^1$, and it need not be an integer.  Its role is to align the mean eigenangle density with the mean density of zeros near the central point.

\subsection{Excision as a deformation of measure}
\label{subsec:excision-measure}

\paragraph{Literature.}
Characteristic polynomials of random matrices as models for values of $L$-functions originate with Keating and Snaith \cite{KeatingSnaith2000}; the derivative statistics $D_A=|P_A'(1)|$ relevant to rank-one curves were studied by Snaith \cite{Snaith2005}.  The excised ensemble is due to Due\~nez, Huynh, Keating, Miller and Snaith \cite{DHKMS2012}, whose motivation is the arithmetic discretization of central values coming from the Waldspurger--Kohnen--Zagier formula \cite{KohnenZagier1981}; the corresponding discussion for odd quadratic twists is in Conrey, Rubinstein, Snaith and Watkins \cite{CRSW2007}.  The point of the present subsection is that the reweighting is a change of measure and nothing more, so the invariance recorded in \cref{prop:conditional-invariance} is elementary but restrictive.

Let $\mathcal M=SO(2N+1)$ with normalized Haar probability measure $\mu_0$, and let
\begin{equation}
  Y(A)=\log D_A,
  \qquad D_A=|P_A'(1)|.
\end{equation}
For a nonnegative integrable weight $W$, define
\begin{equation}
  d\mu_W(A)
  =\frac{W(Y(A))}{\int_{\mathcal M}W(Y(B))\,d\mu_0(B)}\,d\mu_0(A).
  \label{eq:weighted-Haar-measure}
\end{equation}
If $W>0$, the path
\begin{equation}
  d\mu_t(A)
  =\frac{W(Y(A))^t}{\int_{\mathcal M}W(Y(B))^t\,d\mu_0(B)}\,d\mu_0(A),
  \qquad 0\le t\le1,
  \label{eq:measure-homotopy}
\end{equation}
is a continuous interpolation between Haar measure and the reweighted law.  It is legitimate to call this a \emph{homotopy of probability measures}.  It is not a homotopy of the group $SO(2N+1)$, nor a gauge transformation, unless an additional bundle, gauge group and equivalence relation are supplied.  A sharp cutoff, for which $W$ vanishes on a set of positive Haar measure, is excluded from this strictly positive path; it may be approached by soft positive weights but does not belong to the same equivalence class of measures.

The exact invariant of this deformation is measure-theoretic.

\begin{proposition}[Conditional invariance under derivative reweighting]
\label{prop:conditional-invariance}
Let $X$ be any measurable observable on $\mathcal M$, let $Y=\log D_A$, and let $\mu_W$ be defined by \eqref{eq:weighted-Haar-measure}.  Wherever regular conditional distributions are defined,
\begin{equation}
  p_W(x,y)
  =\frac{W(y)p_0(x,y)}{\mathbb E_0W(Y)},
  \qquad
  p_W(x\mid y)=p_0(x\mid y).
  \label{eq:excision-factorization}
\end{equation}
Thus a derivative-only reweighting changes the marginal law of $Y$ but not the conditional law of $X$ at fixed $Y$.
\end{proposition}

\begin{proof}
Disintegrate $d\mu_0$ with respect to $Y$.  Since the Radon--Nikodym factor $W(Y)$ is constant on each fiber $Y=y$, it cancels when the joint density is divided by the reweighted marginal density of $Y$.
\end{proof}

This proposition is the precise content of the conditional test in \cref{sec:dhkms-tests}.  The preserved conditional law is not a topological invariant of the group; it is an invariant of the restricted class of measure changes whose Radon--Nikodym derivative depends only on $Y$.

\subsection{The explicit formula as a torus flow}
\label{subsec:torus-flow}

\paragraph{Literature.}
A truncated explicit formula is a trigonometric polynomial in the logarithmic variable, hence an almost periodic function in the sense of Bohr \cite{Bohr1947}.  The associated linear flow on $\T^M$, and its unique ergodicity when the frequencies are rationally independent, are classical \cite{CFS1982}, \cite{KatokHasselblatt1995}.  For the Morse theory used in \cref{prop:phase-morse} see Milnor \cite{Milnor1963}; for the coarea formula used in \cref{prop:torus-local-time} see Federer \cite{Federer1969} or Evans and Gariepy \cite[Ch.~3]{EvansGariepy1992}.  The shrinking-target problems that \cref{subsec:crossing-topology} reduces to are studied for general flows by Hill and Velani \cite{HillVelani1995} and Kleinbock and Margulis \cite{KleinbockMargulis1999}.

A literal dynamical-topological structure appears after truncating the noncentral explicit formula.  For positive ordinates $\gamma_{E,1},\ldots,\gamma_{E,M}$, write the truncated centred field as
\begin{equation}
  Z_{E,M}(u)
  =\sum_{k=1}^{M}A_{E,k}
   \cos(\gamma_{E,k}u+\delta_{E,k})
  =\Phi_{E,M}(\boldsymbol\gamma_Eu),
  \label{eq:truncated-phase-field}
\end{equation}
where
\begin{equation}
  \Phi_{E,M}(\boldsymbol\phi)
  =\sum_{k=1}^{M}A_{E,k}\cos(\phi_k+\delta_{E,k}),
  \qquad
  A_{E,k}=\frac{2}{\sqrt{\tfrac14+\gamma_{E,k}^2}},
\end{equation}
and
\begin{equation}
  \boldsymbol\gamma_Eu
  =(\gamma_{E,1}u,\ldots,\gamma_{E,M}u)\pmod{2\pi}.
\end{equation}
Hence $Z_{E,M}$ is an observable along the Kronecker flow
\begin{equation}
  \varphi_u:\T^M\longrightarrow\T^M,
  \qquad
  \varphi_u(\boldsymbol\phi)
  =\boldsymbol\phi+u\boldsymbol\gamma_E.
  \label{eq:Kronecker-flow}
\end{equation}
If the ordinates are linearly independent over $\Q$, this flow is uniquely ergodic.  For a regular value $c$ of $\Phi_{E,M}$, the level set
\begin{equation}
  \Sigma_{E,M}(c)
  =\{\boldsymbol\phi\in\T^M:
  \Phi_{E,M}(\boldsymbol\phi)=c\}
  \label{eq:crossing-hypersurface}
\end{equation}
is a codimension-one submanifold.  Crossings of the truncated continuous race are intersections of the orbit of \eqref{eq:Kronecker-flow} with this hypersurface.

\begin{proposition}[Morse structure of the truncated phase observable]
\label{prop:phase-morse}
The function $\Phi_{E,M}$ is Morse on $\T^M$.  Its $2^M$ critical points are indexed by sign vectors $\boldsymbol\varepsilon\in\{\pm1\}^M$, with critical values
\begin{equation}
  c_{\boldsymbol\varepsilon}
  =\sum_{k=1}^{M}\varepsilon_kA_{E,k}.
  \label{eq:Morse-critical-values}
\end{equation}
The Morse index is the number of coordinates for which $\varepsilon_k=+1$.  Consequently, regular level sets are smooth hypersurfaces, and the diffeomorphism type and homology of the sublevel sets
\begin{equation}
  \{\boldsymbol\phi:\Phi_{E,M}(\boldsymbol\phi)\le c\}
\end{equation}
can change only when $c$ crosses a signed amplitude sum \eqref{eq:Morse-critical-values}.
\end{proposition}

\begin{proof}
At a critical point, each $\sin(\phi_k+\delta_{E,k})$ vanishes, so $\cos(\phi_k+\delta_{E,k})=\varepsilon_k\in\{\pm1\}$.  The Hessian is diagonal with entries $-A_{E,k}\varepsilon_k$, all nonzero.  Hence every critical point is nondegenerate, its value is \eqref{eq:Morse-critical-values}, and its index is the number of negative Hessian entries.  Standard Morse theory gives the assertion about the sublevel filtration, allowing simultaneous handle attachments when different sign vectors have the same critical value.
\end{proof}

\begin{proposition}[Finite-dimensional level-set local time]
\label{prop:torus-local-time}
Assume that $\gamma_{E,1},\ldots,\gamma_{E,M}$ are linearly independent over $\Q$, and let $c$ be a regular value of $\Phi_{E,M}$.  With normalized Haar measure $d\boldsymbol\phi$ on $\T^M$,
\begin{align}
 &\lim_{\varepsilon\downarrow0}
  \lim_{U\to\infty}
  \frac{1}{2\varepsilon U}
  \int_0^U
  \mathbf 1_{\{|Z_{E,M}(u)-c|<\varepsilon\}}\,du
  \notag\\
 &\hspace{6em}
  =\int_{\Sigma_{E,M}(c)}
  \frac{d\sigma_c(\boldsymbol\phi)}
  {|\nabla\Phi_{E,M}(\boldsymbol\phi)|},
  \label{eq:torus-coarea-local-time}
\end{align}
where $d\sigma_c$ is the hypersurface measure induced by normalized Haar measure.
\end{proposition}

\begin{proof}
For almost every sufficiently small $\varepsilon>0$, the two boundary levels $c\pm\varepsilon$ are regular, so the boundary of the strip $|\Phi_{E,M}-c|<\varepsilon$ has Haar measure zero.  Unique ergodicity, applied through continuous upper and lower approximations (equivalently, the portmanteau theorem), identifies the long-time occupation proportion with the Haar measure of the strip.  The coarea formula, followed by $\varepsilon\downarrow0$, gives \eqref{eq:torus-coarea-local-time}.
\end{proof}

This is a rigorous finite-dimensional local-time interpretation.  It concerns continuous logarithmic time $u$ and a fixed truncation $M$.  It does not by itself justify sampling only at $u=\log q$, shrinking the window at the prime-dependent rate in \eqref{eq:crossing-criterion}, or passing to the full infinite zero set.  Those are exactly the additional arithmetic and analytic difficulties isolated in \cref{sec:crossings}.  \Cref{tab:topology-crossing-dictionary} records the correspondence term by term, together with the point at which each entry stops being unconditional.

\begin{table}[H]
\centering
\small
\caption{Dictionary between the continuous phase flow and the arithmetic crossing problem.}
\label{tab:topology-crossing-dictionary}
\begin{tabularx}{\textwidth}{@{}P{4.7cm}>{\raggedright\arraybackslash}X@{}}
\toprule
Continuous torus formulation & Prime-sampled arithmetic formulation \\
\midrule
Orbit $\boldsymbol\phi+u\boldsymbol\gamma_E$ & Samples at $u=\log q$ \\
Regular hypersurface $\Sigma_{E,M}(c)$ & Reference level for the elliptic race \\
Fixed strip $|\Phi_{E,M}-c|<\varepsilon$ & Frobenius tube $|B_E^*(q)-c|<T_E(q)$ \\
Width $\varepsilon$ chosen independently & Width $T_E(q)=|a_q(E)|\log q/(2q)$ \\
Unoriented occupation measure & Orientation supplied by $\operatorname{sgn}a_q(E)$ \\
Coarea/unique-ergodicity theorem & Prime-sampled shrinking-target transfer (open) \\
\bottomrule
\end{tabularx}
\end{table}

\subsection{The controlled topological synthesis}
\label{subsec:topological-synthesis}

The legitimate topological content of the framework may therefore be summarized as
\begin{equation}
  \boxed{
  \begin{gathered}
  \operatorname{Frob}_q\text{ on }H^1_{\mathrm{\acute et}}
  \longrightarrow a_q(E)
  \longrightarrow\text{prime-side arithmetic},\\
  \text{family symmetry}
  \longrightarrow SO(\mathrm{odd})\text{ hard edge}
  \longrightarrow\text{low-zero statistics},\\
  \{\gamma_{E,k}\}_{k\le M}
  \longrightarrow\text{Kronecker flow on }\T^M
  \longrightarrow\text{level-set crossings and local time}.
  \end{gathered}}
  \label{eq:topological-synthesis}
\end{equation}
The three lines are exactly the topological block \textbf{(T1)}--\textbf{(T3)} of \cref{tab:epistemic-dictionary}.  The first line is local and cohomological, the second is a family-level symmetry model over $\Q$ and a monodromy theorem in appropriate function-field settings, and the third is a dynamical reformulation of the truncated explicit formula.  This separation permits a genuinely topological reading of the ECRH--BSD framework without identifying global zeros with the eigenvalues of a fixed cohomology group or treating a probabilistic reweighting as a deformation of topology.

\section{Exact zero-product decomposition}
\label{sec:hadamard}

\paragraph{Literature.}
The Hadamard factorization of entire functions of order one, and the Riemann--von Mangoldt count $N_E(T)=O(T\log T)$ used to fix the genus, are standard \cite{IwaniecKowalski2004}; entirety of $\Lambda_E$ is modularity \cite{Wiles1995,BCDT2001}.  The BSD conjecture in the form used below is stated in \cite{BSD1965,Tate1966} and reviewed in \cite{Wiles2006}; finiteness of $\ShaE(E)$ in the analytic rank at most one case follows from Gross--Zagier \cite{GrossZagier1986} and Kolyvagin \cite{Kolyvagin1988}, and squareness of its order under the Cassels--Tate hypotheses from \cite{Cassels1965,PoonenStoll1999}.

Define
\begin{equation}
  F_E(z)=\frac{\Lambda_E(1+z)}{z^r}.
  \label{eq:F}
\end{equation}
The functional equation implies $F_E(-z)=F_E(z)$.  The function is even, entire, and of order one.  Since the zero-counting function is $O(T\log T)$, one has
\begin{equation}
  \sum_{\gamma_E>0}\gamma_E^{-2}<\infty,
\end{equation}
so pairing the zeros at $\alpha$ and $-\alpha$ gives a genus-zero product in $z^2$.

\begin{theorem}[Central leading coefficient and zero product]
Define
\begin{equation}
  J_E=\log\left|\frac{F_E(1/2)}{F_E(0)}\right|.
  \label{eq:J-unconditional}
\end{equation}
Then
\begin{equation}
  \boxed{
  \log\left|\frac{L^{(r)}(E,1)}{r!}\right|
  =\frac14\log N_E
   +\left(r-\frac32\right)\log2
   +\log L\!\left(E,\frac32\right)-J_E.
  }
  \label{eq:hadamard-identity}
\end{equation}
Under ECRH,
\begin{equation}
  \boxed{
  J_E=\sum_{\gamma_E>0}
  \log\!\left(1+\frac{1}{4\gamma_E^2}\right).
  }
  \label{eq:J-ECRH}
\end{equation}
\end{theorem}

\begin{proof}
At the origin,
\begin{equation}
  F_E(0)=N_E^{1/2}(2\pi)^{-1}
  \frac{L^{(r)}(E,1)}{r!},
\end{equation}
whereas
\begin{equation}
  F_E(1/2)=2^r\Lambda_E(3/2).
\end{equation}
Since $\Gamma(3/2)=\sqrt\pi/2$,
\begin{equation}
  \left|\frac{L^{(r)}(E,1)}{r!}\right|
  =2^{r-3/2}N_E^{1/4}|L(E,3/2)|\e^{-J_E}.
  \label{eq:multiplicative-identity}
\end{equation}
Under ECRH,
\begin{equation}
  \frac{F_E(z)}{F_E(0)}
  =\prod_{\gamma_E>0}
   \left(1+\frac{z^2}{\gamma_E^2}\right),
\end{equation}
and evaluation at $z=1/2$ gives \eqref{eq:J-ECRH}.
\end{proof}

For rank one,
\begin{equation}
  \log|L'(E,1)|
  =\frac14\log N_E-\frac12\log2
   +\log|L(E,3/2)|-J_E.
  \label{eq:rank-one}
\end{equation}
Any residual formed by moving these terms to one side vanishes identically.

\section{The endpoint functional and the nonlinear zero remainder}
\label{sec:endpoint}

\paragraph{Literature.}
The even Taylor coefficients of $\log\bigl(F_E(z)/F_E(0)\bigr)$ at the origin are the power sums $\sigma_{E,n}=\sum_{\gamma_E>0}\gamma_E^{-2n}$, that is, the Li coefficients of $L(E,s)$ up to normalization.  Li's criterion \cite{Li1997}, its complements \cite{BombieriLagarias1999}, its Selberg-class form \cite{OmarMazhouda2007} and its automorphic form \cite{Lagarias2007} are therefore the natural context for the functionals $J_E$, $\mathscr V_E$ and $\Delta_E$ introduced here.  What is new below is not the expansion but the fact that a single interior evaluation at $z=1/2$ separates cleanly into $V_E/8$ and a positive remainder.  The boundary value $L(E,3/2)$ sits at the edge of the region of absolute convergence, so its distribution belongs to the circle of edge-value problems studied by Granville and Soundararajan \cite{GranvilleSoundararajan2003}; nonvanishing on that edge is \cite{JacquetShalika1976}.

Let
\begin{equation}
  \Phi_E(t)=\log\frac{F_E(t)}{F_E(0)}
\end{equation}
whenever the real segment from $0$ to $t$ contains no zero of $F_E$.  Define the displaced logarithmic-derivative functional
\begin{equation}
  \mathscr V_E(t)
  =2\Phi_E'(t)
  =2\frac{\Lambda_E'}{\Lambda_E}(1+t)-\frac{2r}{t}.
  \label{eq:displaced-V}
\end{equation}
Then
\begin{equation}
  \boxed{
  J_E=\frac12\int_0^{1/2}\mathscr V_E(t)\,dt.
  }
  \label{eq:integral-identity}
\end{equation}
Under ECRH,
\begin{equation}
  \mathscr V_E(t)
  =4t\sum_{\gamma_E>0}\frac{1}{t^2+\gamma_E^2},
  \qquad
  \mathscr V_E(1/2)=V_E.
  \label{eq:Vt-ECRH}
\end{equation}
Thus a regression of $J_E$ on $V_E$ is a one-point quadrature rule: it compares an integral with the endpoint of its integrand.

Set
\begin{equation}
  j(\gamma)=\log\!\left(1+\frac{1}{4\gamma^2}\right),
  \qquad
  v(\gamma)=\frac{8}{1+4\gamma^2}.
  \label{eq:kernels}
\end{equation}

\begin{proposition}[Exact moment decomposition of $J_E$ and $V_E$]
\label{prop:Delta-decomposition}
Assume ECRH and put
\begin{equation}
  b_\gamma=\frac{1}{1+4\gamma^2},
  \qquad
  S_k(E)=\sum_{\gamma_E>0}b_\gamma^k.
\end{equation}
Then
\begin{equation}
  \boxed{
  V_E=8S_1(E),\qquad
  J_E=\sum_{k\ge1}\frac{S_k(E)}{k}
  =\frac18V_E+\Delta_E,
  }
  \label{eq:Delta-decomposition}
\end{equation}
where
\begin{equation}
  \boxed{
  \Delta_E:=J_E-\frac18V_E
  =\sum_{k\ge2}\frac{S_k(E)}{k}>0.
  }
  \label{eq:Delta-statistic}
\end{equation}
The same identities hold for every finite zero truncation.
\end{proposition}

\begin{proof}
The kernels satisfy $v(\gamma)=8b_\gamma$ and
$j(\gamma)=-\log(1-b_\gamma)=\sum_{k\ge1}b_\gamma^k/k$.  All terms are nonnegative, so Tonelli's theorem permits the interchange of the zero and power sums.  The convergence follows from $S_1(E)=V_E/8<\infty$ and $b_\gamma\le b_{\gamma_{E,1}}<1$.
\end{proof}

\begin{theorem}[First-zero ratio bracket]
Assume ECRH and let $\gamma_{E,1}$ be the first positive zero ordinate.  Put
\begin{equation}
  x_1=\frac{1}{4\gamma_{E,1}^2},
  \qquad
  R(x)=\frac{(1+x)\log(1+x)}{8x}.
\end{equation}
Then
\begin{equation}
  \boxed{
  \frac18\le \frac{J_E}{V_E}\le R(x_1).
  }
  \label{eq:ratio-bracket}
\end{equation}
The upper bound is strictly decreasing with $\gamma_{E,1}$.  Moreover,
\begin{equation}
  J_E\le\frac14V_E
\end{equation}
whenever
\begin{equation}
  \gamma_{E,1}\ge\gamma_*=0.252488\ldots,
\end{equation}
where $x_*=3.921553\ldots$ solves $(1+x)\log(1+x)=2x$ and
$\gamma_*=(2\sqrt{x_*})^{-1}$.
\end{theorem}

\begin{proof}
With $x=(4\gamma^2)^{-1}$,
\begin{equation}
  \frac{j(\gamma)}{v(\gamma)}=R(x).
\end{equation}
The derivative of $R$ has the sign of $x-\log(1+x)>0$, so $R$ increases with $x$ and decreases with $\gamma$, with limit $1/8$ as $\gamma\to\infty$.  Since $J_E/V_E$ is the $v(\gamma_E)$-weighted mean of the ratios $j(\gamma_E)/v(\gamma_E)$, it lies between their infimum and their maximum, the latter occurring at $\gamma_{E,1}$.
\end{proof}

The lower bound in the theorem is also immediate from the positive remainder \eqref{eq:Delta-statistic}.  The upper bracket adds the first-zero control and sharpens, rather than removes, the near-central obstruction.  As $\gamma_{E,1}\to0$,
\begin{equation}
  j(\gamma_{E,1})\sim-2\log(2\gamma_{E,1})\to\infty,
  \qquad
  v(\gamma_{E,1})\to8.
\end{equation}
Thus no uniform affine law $J_E=a+bV_E+o(1)$ can cover the extreme tail, while the ratio bracket remains valid but its upper endpoint degenerates.  A regression slope is not the ratio $J_E/V_E$: the former is a covariance-weighted average of derivative ratios and can lie outside the pointwise ratio interval suggested by a typical first zero.

At $t=1/2$, \eqref{eq:displaced-V} also gives the exact identity
\begin{equation}
  \boxed{
  V_E=\log N_E-2\log(2\pi)
  +2\psi\!\left(\frac32\right)
  +2\frac{L'}{L}\!\left(E,\frac32\right)-4r
  }
  \label{eq:V-Euler}
\end{equation}
under ECRH, with the left side replaced by the corresponding symmetric zero functional off ECRH.

\section{Finite-conductor zero structure and random-matrix comparison}
\label{sec:numerics}

\paragraph{Literature.}
The curves, models and invariants are taken from the LMFDB \cite{LMFDB}, whose elliptic-curve tables follow the algorithms of Cremona \cite{Cremona1997}.  The values $L(E,3/2)$ and $L'(E,1)$ are computed by the smoothed incomplete-gamma method of Dokchitser \cite{Dokchitser2004}; the practical aspects of evaluating $L$-functions and their zeros are surveyed by Rubinstein \cite{RubinsteinLNS2005}.  The random-matrix side uses the Weyl integration formula for the classical compact groups \cite{Weyl1939}, in the form given in \cite{Mehta2004,Forrester2010}, and the symmetry-type predictions of \cite{KatzSarnak1999,KatzSarnak1999BAMS}.

\subsection{Fixed conductor: the exact nonlinear remainder and the first-zero kernel arc}

\paragraph{Literature.}
The isogeny classes, integral models and analytic ranks are those of the LMFDB \cite{LMFDB}.  That the lowest ordinate dominates statistics of this kind at finite conductor is the phenomenon reported by Miller \cite{Miller2006} and analysed by Marshall \cite{Marshall2013}; here it is a consequence of the exact decomposition rather than an empirical observation.

For seventeen rank-one isogeny classes at $N_E=50700$, using all positive zeros through height $20$, define
\begin{equation}
  \Delta_{E,20}=J_{E,20}-\frac18V_{E,20}.
\end{equation}
The exact decomposition \eqref{eq:Delta-decomposition} gives
\begin{equation}
  \operatorname{mean}\Delta_{E,20}=0.035617,
  \qquad
  0.017216\le\Delta_{E,20}\le0.131512.
  \label{eq:Delta-50700-range}
\end{equation}
If
\begin{equation}
  \Delta_{E,1}=j(\gamma_{E,1})-\frac18v(\gamma_{E,1}),
\end{equation}
then the first zero contributes on average $82.4\%$ of $\Delta_{E,20}$, with classwise range $69.1\%$--$97.1\%$.  This is an additive positive decomposition, unlike a covariance-share calculation.

The familiar affine fits remain useful descriptions of the kernel geometry.  One obtains
\begin{equation}
  J_{E,20}=-0.14643+0.17494V_{E,20},
  \qquad R^2=0.9952,
  \label{eq:fixed-regression}
\end{equation}
and for the first zero alone,
\begin{equation}
  j(\gamma_{E,1})=-0.05243+0.17384v(\gamma_{E,1}),
  \qquad R^2=0.9964.
  \label{eq:first-zero-regression}
\end{equation}
The first ordinates lie in
\begin{equation}
  0.583\lesssim\gamma_{E,1}\lesssim1.207,
  \label{eq:first-ordinate-range}
\end{equation}
and the secant slope of the kernel arc
\begin{equation}
  \gamma\longmapsto\bigl(v(\gamma),j(\gamma)\bigr)
\end{equation}
over this interval is $0.17714$.  Removing the first zero gives
\begin{equation}
  J_{E,20}^{\mathrm{rest}}
  =-0.01653+0.13594V_{E,20}^{\mathrm{rest}},
  \qquad R^2=0.9997,
  \label{eq:rest-regression}
\end{equation}
with a slope close to $1/8$ because the higher-power remainder in \eqref{eq:Delta-decomposition} becomes small for high ordinates.

The observed truncated ratios satisfy
\begin{equation}
  0.13044\le \frac{J_{E,20}}{V_{E,20}}\le0.15010,
\end{equation}
and every class lies inside its individual first-zero bracket \eqref{eq:ratio-bracket}.  The appropriate conclusion is therefore
\begin{equation}
  \boxed{
  J_{E,20}=\frac18V_{E,20}+\Delta_{E,20},
  \quad
  \Delta_{E,20}\text{ is predominantly a first-zero statistic.}
  }
  \label{eq:first-zero-conclusion}
\end{equation}
The regression is a visual consequence of this exact decomposition, not a proposed universal affine law.  \Cref{fig:kernel-arc-50700} makes the point graphically: the seventeen pairs $(V_{E,20},J_{E,20})$ are plotted against the translated first-zero kernel arc, and the apparent linearity is accounted for entirely by the shape of $\gamma\mapsto(v(\gamma),j(\gamma))$ over the observed range \eqref{eq:first-ordinate-range} of first ordinates.

\begin{figure}[H]
  \centering
  \includegraphics[width=0.90\textwidth]{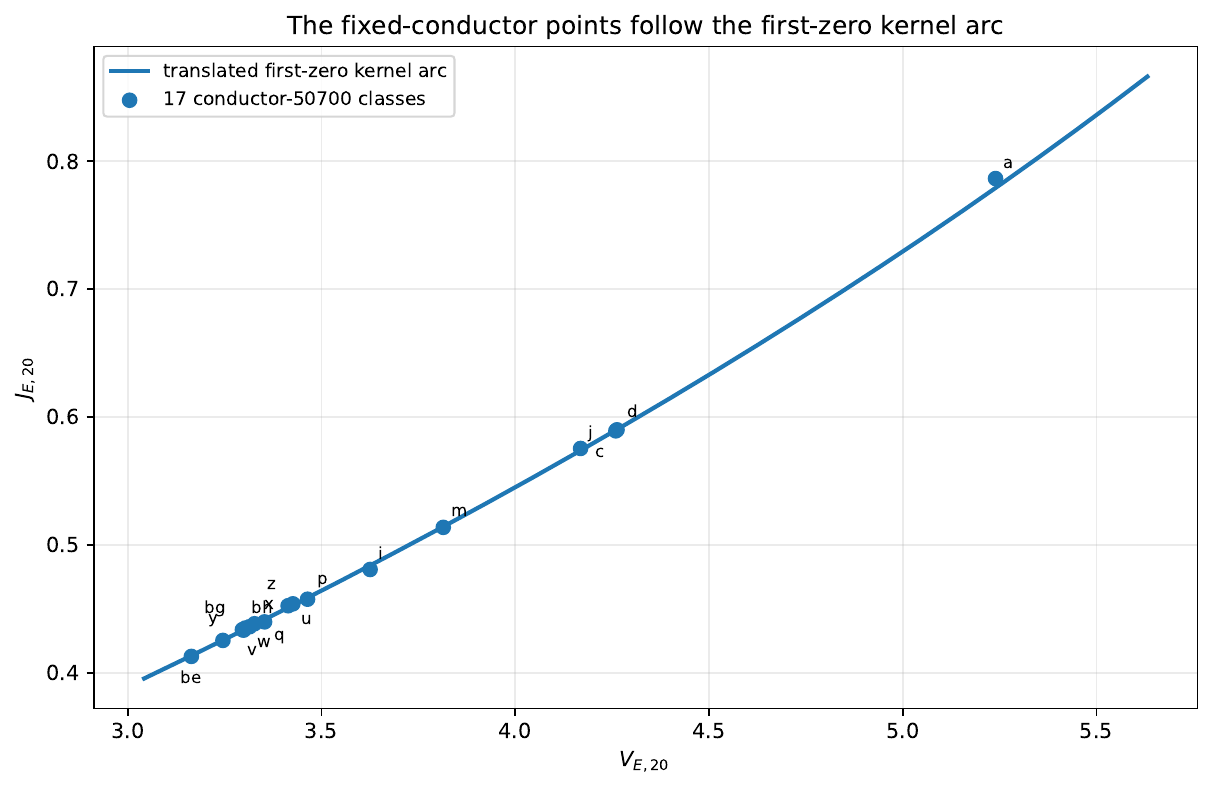}
  \caption{The seventeen total pairs $(V_{E,20},J_{E,20})$ lie close to the translated first-zero kernel arc $\gamma\mapsto(\overline V_{\rm rest}+v(\gamma),\overline J_{\rm rest}+j(\gamma))$.  Labels are the suffixes of the LMFDB isogeny-class names.  The figure displays the deterministic geometric origin of the high-$R^2$ regression.}
  \label{fig:kernel-arc-50700}
\end{figure}

\subsection{Independent edge values and the exact tail beyond height $20$}

For one representative of each isogeny class, $L(E,3/2)$ was evaluated independently from modular coefficients and the functional equation.  Representative integral Weierstrass models were taken from the LMFDB records for conductor $50700$ \cite{LMFDB}.  Put
\begin{equation}
  A=\frac{\sqrt{N_E}}{2\pi}.
\end{equation}
Splitting the Mellin integral at the fixed point of the functional equation gives, for root number $w_E=-1$ and general $s$,
\begin{equation}
\begin{split}
  \Lambda_E(s)=\sum_{n\ge1}a_n\Bigg[
  &\left(\frac{A}{n}\right)^s
  \Gamma\!\left(s,\frac{n}{A}\right)\\
  &+w_E\left(\frac{A}{n}\right)^{2-s}
  \Gamma\!\left(2-s,\frac{n}{A}\right)
  \Bigg].
\end{split}
  \label{eq:incomplete-gamma-L}
\end{equation}
The incomplete-gamma weights decay exponentially, and truncations at $1000$, $1500$, and $2000$ coefficients agree to the displayed precision.

For rank one,
\begin{equation}
  J_E=\frac14\log N_E-\frac12\log2
  +\log|L(E,3/2)|-\log L'(E,1).
  \label{eq:J-from-independent-L32}
\end{equation}
Thus the direct edge value and the independently computed central derivative determine the exact numerical tail
\begin{equation}
  R_E(20)=J_E-J_{E,20}.
\end{equation}
This is an identity-based calibration, not an independent prediction of $J_E$.  Its value is that it compares a smooth zero-density approximation with the tail required by independently evaluated $L$-data and diagnoses the completeness of the finite zero lists.  The resulting edge values and tails are collected in \cref{tab:L32-tail}.

\begin{table}[H]
\centering
\small
\caption{Independent edge values and exact numerical zero-product tails for the seventeen rank-one isogeny classes of conductor $50700$.}
\label{tab:L32-tail}
\begin{tabular}{@{}lrrrr@{}}
\toprule
Class & $L(E,3/2)$ & $J_{E,20}$ & $J_E$ & $R_E(20)$\\
\midrule
50700.a  & 0.543955 & 0.786455 & 0.816730 & 0.030274\\
50700.be & 1.209965 & 0.412732 & 0.443024 & 0.030292\\
50700.bg & 1.154219 & 0.433625 & 0.463909 & 0.030284\\
50700.bh & 1.152257 & 0.434889 & 0.465177 & 0.030288\\
50700.c  & 0.693905 & 0.589109 & 0.619388 & 0.030279\\
50700.d  & 0.697779 & 0.589945 & 0.619600 & 0.029655\\
50700.i  & 0.862426 & 0.480701 & 0.510359 & 0.029658\\
50700.j  & 0.734063 & 0.575352 & 0.605630 & 0.030278\\
50700.m  & 0.821406 & 0.513640 & 0.543921 & 0.030282\\
50700.p  & 0.953600 & 0.457446 & 0.487732 & 0.030286\\
50700.q  & 0.993867 & 0.439720 & 0.470007 & 0.030287\\
50700.u  & 1.082001 & 0.453844 & 0.484144 & 0.030300\\
50700.v  & 1.132197 & 0.433258 & 0.463549 & 0.030291\\
50700.w  & 1.129565 & 0.435927 & 0.466227 & 0.030300\\
50700.x  & 1.130764 & 0.438262 & 0.468551 & 0.030289\\
50700.y  & 1.164768 & 0.425219 & 0.455516 & 0.030296\\
50700.z  & 1.098997 & 0.452354 & 0.482638 & 0.030283\\
\bottomrule
\end{tabular}
\end{table}

The tails separate into two sharply defined groups:
\begin{center}
\small
\begin{tabular}{@{}lrrrr@{}}
\toprule
Listed zeros & Classes & Mean tail & Standard deviation & Range\\
\midrule
35 & 15 & 0.0302873 & $7.8\times10^{-6}$ & $[0.0302743,0.0302998]$\\
36 & 2  & 0.0296565 & $2.1\times10^{-6}$ & $[0.0296550,0.0296580]$\\
\bottomrule
\end{tabular}
\end{center}

\Cref{fig:tail-split-50700} shows the same tails as a scatter against the class index.  The two clusters are separated cleanly, and neither has internal spread comparable to the gap, so the apparent scatter of the tail is a zero-count effect and not numerical noise.

\begin{figure}[H]
  \centering
  \includegraphics[width=0.82\textwidth]{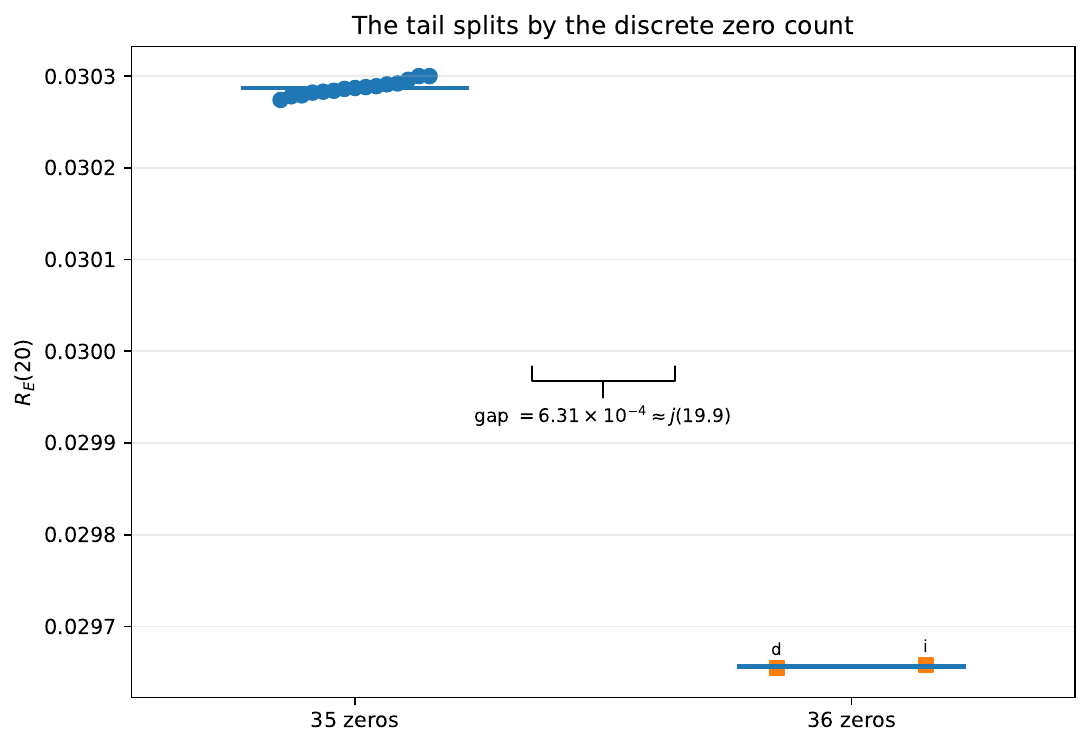}
  \caption{Exact tails beyond height $20$.  The fifteen classes with $35$ listed zeros form a narrow upper cluster; the two classes with $36$ zeros form the lower cluster.  The separation is one kernel unit near the cutoff.}
  \label{fig:tail-split-50700}
\end{figure}

The gap between the group means is $6.307\times10^{-4}$, while
\begin{equation}
  j(19.9)=6.311\times10^{-4}.
\end{equation}
The aggregate standard deviation $2.10\times10^{-4}$ therefore measures the discreteness of the zero-counting function: the two smaller tails are exactly the classes with one additional ordinate near height $20$.  Within each zero-count group the tail is almost constant.

The smooth-density estimate
\begin{equation}
  R_{\mathrm{dens}}(20)=0.0301388
\end{equation}
underestimates the exact numerical tail by about $1.49\times10^{-4}$ for the fifteen $35$-zero classes and overestimates it by about $4.82\times10^{-4}$ for the two $36$-zero classes.  In particular it is not a safe upper bound for a square-isolation argument.  A rigorous tail bound must include an explicit zero-counting remainder.  Heuristically, moving one ordinate across $T$ changes the tail by $j(T)\sim1/(4T^2)$, and the relative size of this unit compared with the leading density tail is
\begin{equation}
  \frac{j(T)}{R_{\mathrm{dens}}(T)}
  \asymp \frac{\pi}{T L_T},
  \qquad
  L_T=\log\frac{\sqrt{N_E}T}{2\pi}+1.
  \label{eq:tail-counting-unit}
\end{equation}
At $T=20$ this is approximately $2.07\%$, matching the observed $35/36$ split.  To turn this into a theorem one needs an explicit bound for the zero-counting error $S_E(T)$; the density integral alone is only a central estimate.

The numerical edge evaluation nevertheless gives a useful completeness diagnostic.  A missing ordinate near $T$ would increase the inferred tail by approximately $j(T)$, and the seventeen values agree within one such zero-counting unit after the $35/36$ split is taken into account.

\paragraph{Data provenance and reproducibility.}
The rank-one sample consists of
\begin{center}
\small
\texttt{50700.a, 50700.u, 50700.v, 50700.bg, 50700.bh, 50700.w, 50700.z,}\\
\texttt{50700.j, 50700.be, 50700.m, 50700.c, 50700.d, 50700.i, 50700.p,}\\
\texttt{50700.q, 50700.x, 50700.y}.
\end{center}
The lookup coordinate is the LMFDB conductor page
\url{https://www.lmfdb.org/EllipticCurve/Q/50700/}; the representative used for each class is the curve with suffix $1$ and the integral model recorded in the accompanying script.  The reproducibility archive contains the point-count table, regenerated ordinates, edge-value checks, figure scripts and curve-level diagnostics.

\subsection{A finite-conductor low-zero repulsion signal}

\paragraph{Literature.}
That elliptic-curve families exhibit more repulsion at the central point than the asymptotic one-level density predicts was observed numerically by Miller \cite{Miller2006} and studied by Marshall \cite{Marshall2013}; lower-order terms accounting for part of the effect are computed in \cite{HKS2009}, and the general one-level machinery is \cite{ILS2000,Miller2004}.

The smooth Weyl density
\begin{equation}
  n_N(u)=\frac1\pi\log\frac{\sqrt N u}{2\pi}
\end{equation}
is not designed to model the hard edge.  Its failure is quantitatively large for the kernel $j$, which heavily weights the smallest ordinates.  At $N=50700$,
\begin{equation}
  \int_{2\pi/\sqrt N}^{20}j(u)n_N(u)\,du=0.9996,
\end{equation}
whereas
\begin{equation}
  \frac1{17}\sum_E J_{E,20}=0.4913.
\end{equation}
The deficit is $0.5082$, or $50.8\%$ of the smooth prediction.  Its localization is shown in \cref{tab:density-bands}.

\begin{table}[H]
\centering
\small
\caption{Contribution of the naive smooth density to the predicted mean of $J_{E,20}$.}
\label{tab:density-bands}
\begin{tabular}{@{}lrr@{}}
\toprule
Ordinate band & Predicted $J$ contribution & Predicted zero count\\
\midrule
$(2\pi/\sqrt N,0.5)$ & 0.4747 & 0.309\\
$(0.5,1)$              & 0.2025 & 0.521\\
$(1,2)$                & 0.1446 & 1.262\\
$(2,5)$                & 0.1096 & 4.583\\
$(5,20)$               & 0.0682 & 28.824\\
\bottomrule
\end{tabular}
\end{table}

All seventeen classes have no positive ordinate below $0.5$; only five have an ordinate below $1$.  Thus the phantom smooth-density contribution below $0.5$, namely $0.4747$, accounts for almost the entire deficit.  The scaled first ordinates
\begin{equation}
  x_{E,1}=\frac{\gamma_{E,1}}{\pi}
  \log\frac{\sqrt N}{2\pi}
\end{equation}
have mean $1.1314$ and range $[0.6638,1.3753]$.  This is a finite sample of seventeen related isogeny classes, not a family theorem, but it makes $J_E$ a sensitive statistic for the low-zero repulsion observed in odd orthogonal elliptic families \cite{Miller2006,Marshall2013}.  Its mean and distribution should be compared with a finite-conductor $SO(\mathrm{odd})$ model rather than with the naive smooth density.

\subsection{Finite-conductor odd-orthogonal comparison and excision}
\label{sec:dhkms-tests}

Due\~nez, Huynh, Keating, Miller and Snaith introduced an excised $SO(2N)$ ensemble for even quadratic twists at finite conductor \cite{DHKMS2012}.  There the cutoff is motivated by arithmetic discretization of central values through the Waldspurger--Kohnen--Zagier formula.  The rank-one odd case has a different status: the regulator or point height enters the derivative formula and varies continuously \cite{CRSW2007}.  An odd-orthogonal derivative cutoff is therefore a formal reweighting unless its arithmetic normalization and lower envelope are specified independently.

If $A\in SO(2N+1)$ has eigenvalues
\begin{equation}
  1,\e^{\pm i\theta_1},\ldots,\e^{\pm i\theta_N},
  \qquad 0<\theta_1\le\cdots\le\theta_N\le\pi,
\end{equation}
and $P_A(z)=\det(I-zA)$, put
\begin{equation}
  D_A:=|P_A'(1)|
  =\prod_{j=1}^{N}(2-2\cos\theta_j).
  \label{eq:reduced-characteristic-derivative}
\end{equation}
Sharp and soft reweightings may be written as in \eqref{eq:sharp-odd-excision}--\eqref{eq:soft-odd-excision}; neither their cutoff nor their softness is supplied by a known odd-rank quantization theorem.
\begin{align}
  \mathcal T_{\tau}^{\mathrm{odd}}
  &=\{A\in SO(2N+1):D_A\ge\tau\},
  \label{eq:sharp-odd-excision}\\
  W_{\tau,\eta}(A)
  &=\left(1+\exp\left[-\frac{\log D_A-\log\tau}{\eta}\right]\right)^{-1}.
  \label{eq:soft-odd-excision}
\end{align}

\paragraph{Finite-conductor convention audit.}
Let
\begin{equation}
  L_E=\log\frac{\sqrt{N_E}}{2\pi}.
\end{equation}
The leading centre-density comparison suggests
\begin{equation}
  x_{E,j}=\frac{L_E}{\pi}\gamma_{E,j},
  \qquad
  x_{A,j}=\frac{N}{\pi}\theta_j,
  \qquad N=L_E.
  \label{eq:unfolding}
\end{equation}
At finite conductor this is only one convention.  The smooth elliptic count is more accurately
\begin{equation}
  N_E^+(T)
  =\frac1\pi\left[T L_E+\Im\log\Gamma(1+iT)\right]-\frac r2+S_E(T),
  \label{eq:exact-smooth-count}
\end{equation}
with a consistently chosen branch and argument remainder $S_E(T)$.  At conductor $50700$ and $T=20$, its smooth part is $35.24$, compared with the observed mean $35.12$; the leading-density integral gives $35.49$.  Because $L_E=3.579\ldots$, matching a slope, a total count, or the local central density changes the effective finite matrix dimension by $O(1)$, enough to alter the first-eigenangle law.  The gamma-factor correction is therefore not an optional decoration but part of the same convention choice as the counting function.

Once a matching rule produces a real effective dimension $N_*$, the matrix-side interpolation is fixed as follows.  Let $n=\lfloor N_*\rfloor$ and $\alpha=N_*-n$.  For the cumulative distribution $F_n^{\mathcal T}$ of any predeclared statistic $\mathcal T$ in Haar $SO(2n+1)$, define
\begin{equation}
  F_{N_*}^{\mathcal T}(x)
  =(1-\alpha)F_n^{\mathcal T}(x)
   +\alpha F_{n+1}^{\mathcal T}(x).
  \label{eq:adjacent-dimension-interpolation}
\end{equation}
This adjacent-integer mixture is a reproducible interpolation convention, not a claim that an orthogonal group of nonintegral dimension exists.  The matching rule and the interpolation are separate choices: the gamma factor helps determine $N_*$, whereas \eqref{eq:adjacent-dimension-interpolation} specifies how a nonintegral value is transferred to integer matrix ensembles.

The finite-conductor null is therefore convention-sensitive and should not be described as parameter-free until these choices are frozen.  A valid comparison must declare one primary exact-count convention, apply \eqref{eq:adjacent-dimension-interpolation}, and report predeclared sensitivity results for alternative density-slope and local-density matchings.

\paragraph{Two nonredundant spectral statistics.}
At conductor $50700$,
\begin{equation}
  \Corr\bigl(J_{E,5},j(\gamma_{E,1})\bigr)=0.9961.
\end{equation}
Thus the first ordinate and the residual product
\begin{equation}
  J_{E,2:5}=\sum_{k=2}^{5}j(\gamma_{E,k})
  \label{eq:residual-product}
\end{equation}
must be tested separately.  The random-matrix residual must be evaluated after the chosen unfolding is inverted to the physical ordinate scale.

\paragraph{Factorization of derivative-dependent reweighting.}
Apply \cref{prop:conditional-invariance} with $X_A=x_1$ and $Y_A=\log D_A$.  Derivative-only excision changes the marginal law of $Y_A$ but not the pointwise conditional law $X_A\mid Y_A=y$.  Aggregate rank correlations are not invariant because the marginal rank map of $Y_A$ changes.

For a specified rank-one curve representative, the BSD-normalized arithmetic quantity is
\begin{equation}
  \mathcal A_E
  :=\frac{L'(E,1)|E(\Q)_{\mathrm{tors}}|^2}
  {\Omega_E\prod_p c_p(E)}
  =\Reg(E)|\ShaE(E)|
  \label{eq:arithmetic-reduced-derivative}
\end{equation}
under BSD.  The individual factors on the left vary within an isogeny class, so the representative and period convention must be recorded.  Lang's height conjecture suggests a family-dependent lower envelope but neither a discrete odd-rank cutoff nor an excision fraction \cite{HindrySilverman1988,Silverman2010}.

\paragraph{Pre-registered experimental protocol.}
A valid test should be fixed before the family zeros are inspected:
\begin{enumerate}[leftmargin=1.6em,itemsep=0.3em]
  \item select a natural non-CM odd-sign rank-one family and certify the analytic rank and the first several positive ordinates;
  \item choose one exact-count matching convention for the elliptic and finite-matrix counting functions, together with predeclared alternative conventions;
  \item test the first-zero PIT and the physical-scale residual product $J_{2:5}$ separately under each frozen null;
  \item construct $\mathcal A_E$ for specified representatives and fix its transformation to the RMT derivative coordinate independently of the zero distributions;
  \item test $x_1\mid\log D_A$ before the derivative marginal or any sharp or soft cutoff;
  \item transfer the complete procedure without refitting to a second non-CM family.
\end{enumerate}
Failure of the conditional law rules out every derivative-only model of the form \eqref{eq:excision-factorization}.  Agreement of the conditional law together with a depleted arithmetic derivative tail supports a selection effect; agreement of both conditional and marginal laws supports an unexcised finite-conductor model under the chosen convention.

\subsection{A fixed-dimensional hard-edge theorem for the zero product}
\label{sec:hard-edge-tail}

\paragraph{Literature.}
The eigenangle density below is the Weyl integration formula for $SO(2N+1)$ \cite{Weyl1939}; see \cite{Mehta2004} or \cite{Forrester2010} for the determinantal form and for the edge exponents that produce the constant $3/2$.  The correspondence between the forced eigenvalue at $1$ and a forced central zero of odd order is the odd-orthogonal half of the Katz--Sarnak dictionary \cite{KatzSarnak1999,KatzSarnak1999BAMS}.

The singularity of $j(t)$ makes the large-product tail analytically accessible.  On the positive eigenangles of Haar $SO(2N+1)$, the Weyl density may be written
\begin{equation}
  p_N(\boldsymbol\theta)
  =\frac1{Z_N}\prod_{j=1}^{N}(1-\cos\theta_j)
   \prod_{1\le j<k\le N}(\cos\theta_j-\cos\theta_k)^2,
  \qquad \boldsymbol\theta\in(0,\pi)^N,
  \label{eq:SOodd-Weyl-density}
\end{equation}
where $Z_N$ absorbs the normalization and the symmetry factor.  Define the physical-angle zero-product statistic
\begin{equation}
  \mathcal J_N=\sum_{j=1}^{N}j(\theta_j),
  \qquad j(t)=\log\!\left(1+\frac{1}{4t^2}\right).
  \label{eq:RMT-zero-product}
\end{equation}

\begin{theorem}[Odd-orthogonal hard-edge tail]
For every fixed $N\ge1$, there is a finite constant $C_N>0$ such that
\begin{equation}
  \Pr\{\mathcal J_N>y\}
  \sim C_N\e^{-3y/2},
  \qquad y\to\infty.
  \label{eq:hard-edge-tail-law}
\end{equation}
More explicitly, with $\mathbf u=(u_2,\ldots,u_N)$,
\begin{align}
  C_N&=\frac{N}{48Z_N}\int_{(0,\pi)^{N-1}}
  A_N(\mathbf u)
  \exp\!\left(\frac32\sum_{k=2}^{N}j(u_k)\right)d\mathbf u,
  \label{eq:hard-edge-constant}\\
  A_N(\mathbf u)&=
  \prod_{k=2}^{N}(1-\cos u_k)^3
  \prod_{2\le j<k\le N}(\cos u_j-\cos u_k)^2.
  \label{eq:hard-edge-boundary-density}
\end{align}
For $N=1$, the integral is interpreted as $1$.
\end{theorem}

\begin{proof}
Distinguish one eigenangle $t$ and write the others as $\mathbf u$.  From \eqref{eq:SOodd-Weyl-density},
\begin{equation}
  p_N(t,\mathbf u)
  =\frac{t^2}{2Z_N}A_N(\mathbf u)(1+o(1)),
  \qquad t\downarrow0,
  \label{eq:Weyl-boundary-expansion}
\end{equation}
for almost every $\mathbf u$.  At the same time,
\begin{equation}
  j(t)=-2\log(2t)+o(1).
  \label{eq:j-hard-edge-asymptotic}
\end{equation}
Put $R(\mathbf u)=\sum_{k=2}^{N}j(u_k)$.  For fixed $\mathbf u$, the boundary of
$j(t)+R(\mathbf u)>y$ is
\begin{equation}
  t_y(\mathbf u)
  =\frac1{2\sqrt{\exp(y-R(\mathbf u))-1}}
  \sim\frac12\exp\!\left[-\frac{y-R(\mathbf u)}2\right].
\end{equation}
Integrating \eqref{eq:Weyl-boundary-expansion} from $0$ to $t_y(\mathbf u)$ gives
\begin{equation}
  \frac{A_N(\mathbf u)}{48Z_N}
  \exp\!\left(\frac32R(\mathbf u)\right)
  \e^{-3y/2}(1+o(1)).
\end{equation}
The boundary integral is finite.  Indeed, as one of the remaining angles tends to zero,
$(1-\cos u)^3\asymp u^6$ while $\exp(3j(u)/2)\asymp u^{-3}$, leaving an integrable factor $u^3$; eigenangle collisions only introduce additional zeros through the squared Vandermonde.  Dominated convergence therefore applies.  Summing over the $N$ possible distinguished eigenangles gives \eqref{eq:hard-edge-constant}.  Regions in which two or more angles approach the hard edge on the same exponential scale are of lower order because of the squared Vandermonde factor, so they do not alter the leading constant.
\end{proof}

The exponent $3/2$ is universal at fixed odd-orthogonal dimension: quadratic eigenangle repulsion gives $\Pr(\theta_{\min}<\varepsilon)\asymp\varepsilon^3$, while the logarithmic kernel converts $\varepsilon$ into $\e^{-y/2}$.  To compare with elliptic ordinates under the density-slope convention $\gamma=(N/L_E)\theta$, one should use
\begin{equation}
  \mathcal J_{N,L_E}
  =\sum_{j=1}^{N}\log\!\left(1+\frac{L_E^2}{4N^2\theta_j^2}\right).
\end{equation}
Put $a=L_E/N$ and $j_a(t)=\log(1+a^2/(4t^2))$.  The exponent remains $3/2$, and the exact scaled constant is
\begin{equation}
  C_N(a)=\frac{Na^3}{48Z_N}
  \int_{(0,\pi)^{N-1}}A_N(\mathbf u)
  \exp\!\left(\frac32\sum_{k=2}^{N}j_a(u_k)\right)d\mathbf u.
  \label{eq:scaled-hard-edge-constant}
\end{equation}
Thus the small-angle threshold contributes $a^3$, while the remaining boundary integral is also rescaled.  The exponent is scale-invariant but the finite-dimensional constant is convention-dependent.  This theorem makes the first-zero dominance of extreme zero products precise.  Further work is needed to compute these constants, determine the finite-$N$ crossover for conductor mixtures, and derive the corresponding centred law after unfolding to elliptic ordinates.

\subsection{Cross-conductor CM twists: a descriptive kernel check}

\paragraph{Literature.}
For complex multiplication and its effect on the Frobenius angles see Silverman \cite{Silverman2009}.  CM families do not have the same low-zero symmetry type as the non-CM families modelled in \cref{subsec:family-symmetry} \cite{KatzSarnak1999,ILS2000}, so what follows is a descriptive kernel check and nothing more; no family conclusion is transported from it.

For forty-nine numerical rank-one CM prime twists, the raw regressions
\begin{align}
  J_E&=-0.25146+0.19184V_E,&&R^2=0.9441,\label{eq:CM-full-kernel}\\
  j_1(E)&=-0.20679+0.24147v_1(E),&&R^2=0.9584,\label{eq:CM-first-kernel}\\
  J_E-j_1(E)&=-0.01576+0.13387\bigl(V_E-v_1(E)\bigr),&&R^2=0.9987\label{eq:CM-rest-kernel}
\end{align}
remain useful descriptive checks of the same kernel geometry.  In particular, the residual-spectrum slope approaches the exact high-zero coefficient $1/8$ in \eqref{eq:Delta-decomposition}.

A cross-conductor regression based on
\begin{equation}
  \widetilde J_E=J_E-\frac14\log N_E,
  \qquad
  \widetilde V_E=V_E-\log N_E
  \label{eq:centered-variables}
\end{equation}
is not interpreted here: using different conductor subtractions in the dependent and explanatory variables mechanically creates a residual $\log N_E$ coefficient.  Comparisons between $\log N_E$ and $\log\log N_E$ in that mismatched regression are therefore discarded.

Moreover, these forty-nine twists are CM.  They may illustrate the deterministic $j$--$v$ relation, but they do not provide a non-CM orthogonal symmetry test, and the non-CM symmetric-power input invoked for the race mean should not be transferred to them without a separate CM treatment.  Any family-level finite-conductor centring must instead estimate coherent mean functions for $J_E$, $V_E$, and $\Delta_E$ in a substantially wider, natural non-CM family.

\subsection{The family-fluctuation problem}

\paragraph{Literature.}
The moment and value-distribution predictions that a family statement would have to match are those of Keating and Snaith \cite{KeatingSnaith2000} and of Conrey, Farmer, Keating, Rubinstein and Snaith \cite{CFKRS2005}; the edge-value analogue, with the cancellation that a naive independent-prime model misses, is Granville and Soundararajan \cite{GranvilleSoundararajan2003}.

The exact identity can be written
\begin{equation}
  \log\left|\frac{L^{(r)}(E,1)}{r!}\right|
  =\left(r-\frac32\right)\log2
   +\log|L(E,3/2)|-\widetilde J_E.
  \label{eq:centered-identity}
\end{equation}
For a non-CM family $\mathscr C$ of fixed root number, define finite-conductor centred variables
\begin{equation}
  X_E=\log|L(E,3/2)|-m_L(N_E),
  \qquad
  Y_E=\widetilde J_E-m_J(N_E),
  \label{eq:XY}
\end{equation}
where $m_L$ and $m_J$ incorporate the lower-order family mean.  The centred leading coefficient is $X_E-Y_E$ up to the corresponding deterministic mean.

\begin{problem}[Joint boundary-value--zero-product law]
Determine the joint finite-conductor and limiting laws of $(X_E,Y_E)$ in non-CM orthogonal families.  In particular:
\begin{enumerate}
  \item determine the marginal law of the boundary value $L(E,3/2)$;
  \item compare the marginal law of $Y_E$ with an orthogonal random-matrix model for the zero product;
  \item determine the covariance and its dependence on conductor, symmetry type, and the first zero;
  \item isolate the small-leading-coefficient tail where the first zero approaches the centre.
\end{enumerate}
\end{problem}

The word ``boundary'' is important.  In the normalization \eqref{eq:completed}, the Euler product is absolutely convergent for $\Re s>3/2$, not at $s=3/2$.  Thus the assertion that $L(E,3/2)$ is an immediately solvable absolutely convergent random Euler product is incorrect.  The point corresponds to the edge $\Re s=1$ in unitary normalization.  Random Euler-product methods remain relevant by analogy with edge-value problems \cite{GranvilleSoundararajan2003}, but they require cancellation and edge nonvanishing rather than a direct independent-prime product argument.

\section{BSD information in truncated zero products}
\label{sec:bsd-information}

\paragraph{Literature.}
The conjecture in the refined form used here is \cite{BSD1965}, with the modern formulation in \cite{Tate1966} and a survey in \cite{Wiles2006}.  For analytic rank at most one the group $E(\Q)$ has the predicted rank and $\ShaE(E)$ is finite, by Gross--Zagier \cite{GrossZagier1986} and Kolyvagin \cite{Kolyvagin1988}; under the Cassels--Tate hypotheses the order of $\ShaE(E)$ is a square \cite{Cassels1965,PoonenStoll1999}, which is the arithmetic input that makes the interval \eqref{eq:square-interval} informative.  Periods, Tamagawa factors, torsion and regulators are computed as in \cite{Cremona1997} and tabulated in \cite{LMFDB}; the central derivative is evaluated as in \cite{Dokchitser2004}.  The zero-counting remainder needed to certify the tail is the elliptic case of the standard argument-principle estimate \cite{IwaniecKowalski2004}.

The refined BSD formula is
\begin{equation}
  \frac{L^{(r)}(E,1)}{r!}
  =\frac{\Omega_E\Reg(E)|\ShaE(E)|\prod_p c_p(E)}
  {|E(\Q)_{\mathrm{tors}}|^2}.
  \label{eq:BSD}
\end{equation}
Substitution of the full product $J_E$ from \eqref{eq:multiplicative-identity} merely rewrites the standard analytic value of $|\ShaE(E)|$.  A nontrivial information question begins only after truncating the zero product.

Assume ECRH and put
\begin{equation}
  J_E(T)=\sum_{0<\gamma_E\le T}
  \log\!\left(1+\frac{1}{4\gamma_E^2}\right),
  \qquad
  R_E(T)=J_E-J_E(T)\ge0.
  \label{eq:J-truncation}
\end{equation}
Define
\begin{equation}
  \mathscr D_E(T)
  =\frac{2^{r-3/2}N_E^{1/4}L(E,3/2)\e^{-J_E(T)}
  |E(\Q)_{\mathrm{tors}}|^2}
  {\Omega_E\Reg(E)\prod_p c_p(E)}.
  \label{eq:D-truncated}
\end{equation}
Then BSD and the Hadamard identity give
\begin{equation}
  \mathscr D_E(T)=|\ShaE(E)|\e^{R_E(T)}.
  \label{eq:D-tail}
\end{equation}
If $0\le R_E(T)\le\varepsilon_E(T)$, then
\begin{equation}
  \boxed{
  \mathscr D_E(T)\e^{-\varepsilon_E(T)}
  \le |\ShaE(E)|\le\mathscr D_E(T).
  }
  \label{eq:square-interval}
\end{equation}
When $\ShaE(E)$ is finite, its order is a square under the usual Cassels--Tate hypotheses \cite{Cassels1965,PoonenStoll1999}.  The first ordinates therefore carry a quantifiable amount of arithmetic information: the interval width is the multiplicative factor $\e^{\varepsilon_E(T)}$.

A rigorous upper bound may be written in Stieltjes form,
\begin{equation}
  R_E(T)
  =\int_T^\infty j(u)\,dN_E^+(u)
  =-j(T)N_E^+(T)-\int_T^\infty N_E^+(u)j'(u)\,du,
  \label{eq:tail-bound}
\end{equation}
with the endpoint convention chosen consistently.  An effective theorem requires an explicit zero-counting formula and a bound for its remainder.  The leading-density approximation
\begin{equation}
  R_E(T)\approx
  \int_T^\infty
  \log\!\left(1+\frac{1}{4u^2}\right)
  \frac1\pi\log\frac{\sqrt{N_E}u}{2\pi}\,du
  \label{eq:tail-heuristic}
\end{equation}
is a useful central estimate, but the conductor-$50700$ audit shows that it cannot be used as an upper bound without a zero-counting correction.

\begin{proposition}[From zero-counting error to a certified product tail]
\label{prop:counting-to-tail}
Suppose that for $u\ge T$ one has
\begin{equation}
  N_E^+(u)=M_E(u)+S_E(u),
  \qquad |S_E(u)|\le B_E(u),
\end{equation}
where $M_E$ is an explicit smooth counting function, $B_E$ is nonnegative, and $j(u)B_E(u)\to0$ as $u\to\infty$.  Put
\begin{equation}
  R_{E,\mathrm{main}}(T)=\int_T^\infty j(u)\,dM_E(u).
\end{equation}
Then
\begin{equation}
  \boxed{
  R_E(T)\le R_{E,\mathrm{main}}(T)
  +j(T)B_E(T)
  +\int_T^\infty B_E(u)|j'(u)|\,du .
  }
  \label{eq:certified-tail-from-counting}
\end{equation}
The analogous lower bound is obtained by reversing the sign of the final two terms.
\end{proposition}

\begin{proof}
Write $R_E(T)=\int_T^\infty j\,dM_E+\int_T^\infty j\,dS_E$.  Integration by parts gives
\[
  \int_T^\infty j\,dS_E
  =-j(T)S_E(T)-\int_T^\infty S_E(u)j'(u)\,du,
\]
and the asserted bounds follow from $|S_E|\le B_E$.
\end{proof}

\subsection{What the truncated product can and cannot do}

\paragraph{Literature.}
Direct evaluation of $L^{(r)}(E,1)$ by the smoothed method of Dokchitser \cite{Dokchitser2004}, with the practical refinements described by Rubinstein \cite{RubinsteinLNS2005}, is computationally more natural than locating enough zeros for the same purpose; the arithmetic invariants entering \eqref{eq:BSD} are obtained as in Cremona \cite{Cremona1997}.  The construction below should therefore be read as an information statement, not as an algorithm.

The construction should not be advertised as a faster algorithm for $|\ShaE(E)|$.  The edge value, period, regulator, Tamagawa factors and torsion already locate the analytic BSD value, and for large square order all of these multiplicative factors must be known to relative precision comparable with the spacing of adjacent squares.  Standard direct evaluations of the central leading coefficient are more natural computationally than locating enough zeros solely for this purpose.

Two interpretations remain meaningful:
\begin{enumerate}
  \item \textbf{Arithmetic information content.}  The first $m$ ordinates reduce the allowed multiplicative uncertainty from $\e^{J_E}$ to $\e^{R_E(T_m)}$.  The resolving-capacity curve measures how many square candidates remain after revealing part of the spectrum.
  \item \textbf{Zero-list verification.}  Given independently evaluated arithmetic data, the square constraint and the exact Hadamard identity can detect missing or spurious ordinates.  The $35/36$ tail split at conductor $50700$ is a concrete instance of this diagnostic use.
\end{enumerate}
These are ECRH--BSD consistency statements, not an independent proof of BSD or a computational shortcut.

\subsection{Conductor $50700$: an in-sample information audit}

\paragraph{Literature.}
The non-spectral factors are taken from \cite{LMFDB,Cremona1997}, and the squareness of $|\ShaE(E)|$ that turns a multiplicative interval into a finite candidate list is \cite{Cassels1965,PoonenStoll1999}.

For a candidate square $|\ShaE(E)|=n^2$, a multiplicative uncertainty $\varepsilon_m$ distinguishes $n^2$ from the next square whenever
\begin{equation}
  \varepsilon_m<2\log\!\left(1+\frac1n\right).
  \label{eq:square-resolution-criterion}
\end{equation}
Using the exact numerical tails obtained from $L(E,3/2)$ and $L'(E,1)$ gives the capacities in \cref{tab:sieve-capacity}.

\begin{table}[H]
\centering
\small
\caption{In-sample information capacity of the conductor-$50700$ zero lists.  Because the exact tails use the central derivative, this table is a calibration and zero-list audit, not an independent determination of $|\ShaE(E)|$.}
\label{tab:sieve-capacity}
\begin{tabular}{@{}rrrr@{}}
\toprule
$m$ & $\max_E R_{E,m}$ & $n_{\max}$ & $n_{\max}^2$\\
\midrule
0  & 0.816730 & 1  & 1\\
1  & 0.299528 & 6  & 36\\
2  & 0.209395 & 9  & 81\\
3  & 0.160527 & 11 & 121\\
5  & 0.116303 & 16 & 256\\
8  & 0.085378 & 22 & 484\\
12 & 0.064920 & 30 & 900\\
20 & 0.045377 & 43 & 1849\\
35 & 0.030300 & 65 & 4225\\
\bottomrule
\end{tabular}
\end{table}

For the target squares $1,4,9,16$, the number of classes whose in-sample interval is unique before any zero is $17/17$, $16/17$, $13/17$, and $1/17$, respectively; after the first positive zero all seventeen are unique.  These figures quantify information carried by the low spectrum once the non-spectral prefactor is already known.  They must not be read as a computation of $\ShaE(E)$ from the zeros.

\paragraph{A worked information contraction.}
For class $50700.\mathrm a$, the exact value is $J_E=0.816730$.  Revealing the first $m$ ordinates reduces the remaining multiplicative uncertainty from $\exp(J_E)$ to $\exp(R_{E,m})$ as follows:
\begin{center}
\small
\begin{tabular}{@{}rrrrrr@{}}
\toprule
$m$ & $0$ & $1$ & $3$ & $8$ & $35$\\
\midrule
$R_{E,m}$ & $0.816730$ & $0.264914$ & $0.157456$ & $0.085246$ & $0.030276$\\
$\exp(R_{E,m})$ & $2.2631$ & $1.3033$ & $1.1705$ & $1.0890$ & $1.0307$\\
\bottomrule
\end{tabular}
\end{center}
This trace makes the sieve architecture concrete: one low zero removes most of the multiplicative ambiguity, while later zeros refine the interval more slowly.  Because the displayed remainders use the independently evaluated central derivative, the example is an in-sample information audit, not a prediction of $|\ShaE(E)|$.

The measured tails in this table are circular for a prospective BSD determination, since they were inferred through \eqref{eq:J-from-independent-L32} using $L'(E,1)$.  A non-circular experiment must either use a certified upper bound derived only from zero counting, or calibrate a bound on training curves and apply it to held-out curves.  The smooth-density value by itself is unsafe because it underestimates fifteen of the seventeen exact numerical tails.

\begin{problem}[Certified spectral information]
Construct an explicit upper bound for $R_E(T)$ using a zero-counting remainder, calibrate any numerical constants on a training family, and test on held-out curves whether truncated zero products correctly bound the independently evaluated BSD leading coefficient and detect incomplete zero lists.
\end{problem}

\section{Intrinsic elliptic analogues of false Chebyshev primes}
\label{sec:crossings}

The construction is motivated by the false Chebyshev-prime transition layer introduced in \cite{PlanatFalseCheb2026}.  There the relevant condition is controlled by the midpoint of the jump of the classical Chebyshev function, and the natural asymptotic observable is a logarithmically weighted local time.  The present elliptic analogue does not use that classical set as an external selector of curves.  It transfers only the local mechanism: a prime's own Frobenius jump can carry a cumulative race across a reference level, producing a curve-dependent crossing set.

The topological interpretation of \cref{sec:topological-interpretation} upgrades this half-jump mechanism to a prime-sampled shrinking-target problem on the phase torus generated by the noncentral zeros.

\paragraph{Literature.}
The classical race whose jumps are being imitated is the one analysed by Rubinstein and Sarnak \cite{RubinsteinSarnak1994}, and its elliptic version by Fiorilli \cite{Fiorilli2014}; the notion of a Chebyshev prime, defined through the jumps of $\mathrm{li}(x)-\pi(x)$ and its Chebyshev-function variants, was introduced in \cite{PlanatSole2013}, and the half-jump mechanism together with the weighted local-time law it obeys are from \cite{PlanatFalseCheb2026}.  Once the level condition is written on the phase torus, the object is a shrinking-target set for a linear flow, sampled along the sparse sequence $u=\log q$.  For shrinking targets under unrestricted time see Hill and Velani \cite{HillVelani1995} and Kleinbock and Margulis \cite{KleinbockMargulis1999}; the difficulty here, and the reason no theorem is claimed, is that arithmetic sampling destroys the equidistribution input on which those results rest.

\subsection{Midpoint, half-jump, and orientation}

\paragraph{Literature.}
The midpoint convention is forced rather than chosen: a truncated explicit formula converges at a jump to the mean of the one-sided limits, as in the classical treatment of $\psi_0$ \cite{IwaniecKowalski2004}.  The half-jump mechanism in the classical case is \cite{PlanatFalseCheb2026}, and the race whose sign changes are being tracked is the elliptic analogue of \cite{RubinsteinSarnak1994,Fiorilli2014}.

At a good prime $q$, define the midpoint of the logarithmically weighted race \eqref{eq:race} by
\begin{equation}
  B_E^{*}(q)=\frac{1}{\sqrt q}
  \left(
  -\sum_{\substack{\ell<q\\\ell\nmid N_E}}
   \frac{a_\ell(E)}{\sqrt\ell}\log\ell
  -\frac{a_q(E)\log q}{2\sqrt q}
  \right).
  \label{eq:midpoint}
\end{equation}
Set
\begin{equation}
  H_E(q)=\frac{a_q(E)\log q}{2q},
  \qquad T_E(q)=|H_E(q)|.
  \label{eq:half-jump}
\end{equation}
The pre- and post-jump values at the same normalization are
\begin{equation}
  \mathcal E_E^{*}(q^-)=B_E^{*}(q)+H_E(q),
  \qquad
  \mathcal E_E^{*}(q)=B_E^{*}(q)-H_E(q).
  \label{eq:endpoints}
\end{equation}
Therefore the prime jump crosses a level $c$ if and only if
\begin{equation}
  \boxed{|B_E^{*}(q)-c|<T_E(q).}
  \label{eq:crossing-criterion}
\end{equation}
The correctly oriented sets are
\begin{equation}
  \mathcal F^{\uparrow}_{E,c}
  =\{q:\ a_q(E)<0,\ |B_E^{*}(q)-c|<T_E(q)\},
  \label{eq:upward}
\end{equation}
\begin{equation}
  \mathcal F^{\downarrow}_{E,c}
  =\{q:\ a_q(E)>0,\ |B_E^{*}(q)-c|<T_E(q)\}.
  \label{eq:downward}
\end{equation}
Their union $\mathcal C_E(c)$ is the full window in \eqref{eq:crossing-criterion}.  The midpoint is not a discretionary convention: a symmetric Fourier explicit formula converges at a jump to the midpoint.  Accordingly, the half-jump enters the spectral model only through the width $T_E(q)$ and the orientation $\operatorname{sgn}a_q(E)$; it must not be added once more to the Fourier field.

Two levels are natural: $c=0$ detects a sign change of the uncentred race, while $c=\mu_E=2r-1$ removes the deterministic mean and probes the centred noncentral field.  An elliptic crossing prime is generated by the interaction of the zero background and the Frobenius amplitude and orientation.

\subsection{Topological meaning of the elliptic crossing set}
\label{subsec:crossing-topology}

\paragraph{Literature.}
The level hypersurfaces and their coarea measure are those of \cref{prop:phase-morse,prop:torus-local-time}, hence \cite{Milnor1963,Federer1969}.  Sets of times at which an orbit hits a shrinking neighbourhood of a fixed target are the subject of the shrinking-target theory of Hill and Velani \cite{HillVelani1995} and of the logarithm laws of Kleinbock and Margulis \cite{KleinbockMargulis1999}; the present targets shrink at a prime-dependent rate and are sampled only at $u=\log q$, which is what places the statement outside that theory.

The analogy with false Chebyshev primes is not limited to the
formal presence of a half-jump.  It has a finite-dimensional
dynamical interpretation through the phase flow of the truncated
explicit formula.

Fix $M\geq 1$ and write the centred $M$-zero truncation in the form
\begin{equation}
  \mathcal E^{*}_{E,M}(\e^u)
  =
  \mu_E+\Phi_{E,M}\bigl(\boldsymbol\theta_E(u)\bigr),
  \qquad
  \boldsymbol\theta_E(u)
  =
  u(\gamma_{E,1},\ldots,\gamma_{E,M})
  \pmod{2\pi},
  \label{eq:crossing-phase-flow}
\end{equation}
where $\Phi_{E,M}$ is the corresponding trigonometric observable on
$\T^M$.  For a regular level $c$, let
\begin{equation}
  \Sigma_{E,M}(c)
  =
  \left\{
  \boldsymbol\theta\in\T^M:
  \mu_E+\Phi_{E,M}(\boldsymbol\theta)=c
  \right\}.
  \label{eq:crossing-level-hypersurface}
\end{equation}
This is a codimension-one level hypersurface of the phase torus.

At a good prime $q$, the exact midpoint condition
\[
  |B_E^{*}(q)-c|<T_E(q)
\]
states that the prime's own Frobenius jump straddles the level $c$.  A symmetric Fourier truncation already models this midpoint.  Thus the prime-sampled phase point
\[
  \boldsymbol\theta_E(\log q)
\]
lies in the oriented, untranslated shrinking tube
\begin{equation}
  \mathcal U_{E,M}(q;c)
  =
  \left\{
  \boldsymbol\theta\in\T^M:
  \left|
  \mu_E+\Phi_{E,M}(\boldsymbol\theta)-c
  \right|<|H_E(q)|
  \right\}
  \label{eq:prime-dependent-tube}
\end{equation}
of the level hypersurface, up to the prime-power, zero-tail, and explicit-formula remainders.  The half-jump determines the tube width and orientation but does not translate its centre.
  Since
\[
  T_E(q)=\frac{|a_q(E)|\log q}{2q}\longrightarrow0,
\]
the elliptic crossing primes form a prime-sampled shrinking-target
set around a level hypersurface of the zero-generated phase flow.

The Frobenius trace carries two distinct pieces of information.
Its magnitude $|a_q(E)|$ determines the width of the shrinking tube,
whereas its sign determines the orientation of the crossing:
$a_q(E)<0$ gives an upward crossing and $a_q(E)>0$ a downward
crossing.  Thus $\mathcal F^{\uparrow}_{E,c}$ and
$\mathcal F^{\downarrow}_{E,c}$ are oriented arithmetic
intersection sets.

This interpretation explains the weighted statistic
\[
  \mathcal L_E(X;c)
  =
  \frac{1}{\log X}
  \sum_{\substack{q\leq X\\q\in\mathcal C_E(c)\\a_q(E)\neq0}}
  \frac{1}{|a_q(E)|}.
\]
The factor $1/|a_q(E)|$ removes the variable Frobenius amplitude,
leaving the logarithmic prime weight associated with occupation of
a shrinking tube.  Under the joint prime-sampled local-limit
hypothesis, the limiting value is the coarea density of the level
$c$.  The continuous torus-flow identity of
\cref{sec:topological-interpretation} supplies the finite-dimensional
topological model; the passage to prime sampling and shrinking
targets is the genuinely arithmetic open problem.

\Cref{fig:crossing-definition} displays the four ingredients at once: the
oriented jump, the shrinking window, the level hypersurface and the
prime-sampled orbit.  Consequently, the elliptic analogue of a false
Chebyshev prime may be summarized as
\begin{center}
\fbox{\parbox{0.86\textwidth}{\centering
A prime whose own Frobenius jump produces an oriented shrinking-target
intersection with a level hypersurface of the zero-generated phase flow.}}
\end{center}

\begin{figure}[htbp]
\centering
\includegraphics[width=\textwidth]{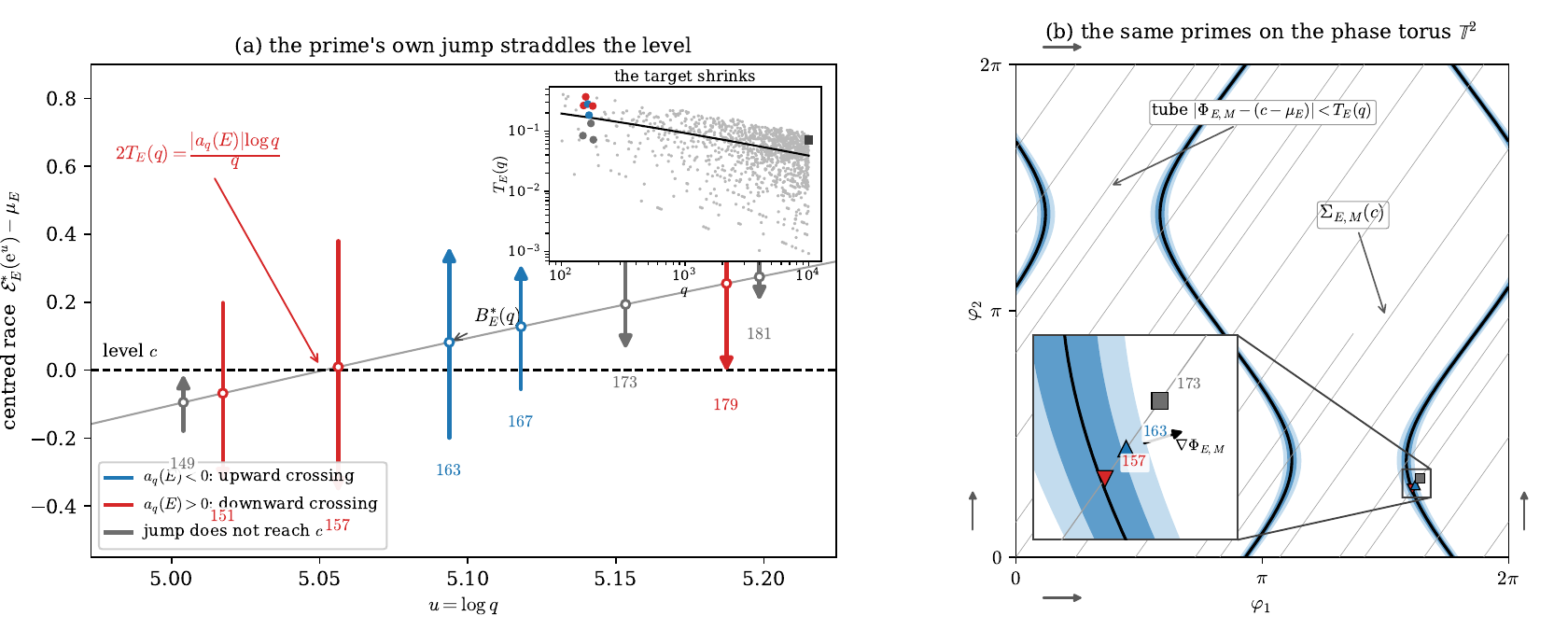}
\caption{The four ingredients of the crossing condition, displayed together.
\textbf{(a)} The centred race near $q\approx160$.  Each vertical arrow is the
jump of $\mathcal E_E^{*}$ at a single prime; it runs from the pre-jump value
$B_E^{*}(q)+H_E(q)$ to the post-jump value $B_E^{*}(q)-H_E(q)$, so its total
length is the window width $2T_E(q)=|a_q(E)|\log q/q$, and its direction is
fixed by $\operatorname{sgn}a_q(E)$.  The prime crosses the level $c$ exactly
when its own arrow reaches across the dashed line, which is condition
\eqref{eq:crossing-criterion}; the open circle is the midpoint $B_E^{*}(q)$ to
which the truncated explicit formula converges.  Inset: $T_E(q)$ for all good
$q\le10^4$, with the Sato--Tate mean $\tfrac{4}{3\pi}\log q/\sqrt q$; the
target shrinks like $q^{-1/2}\log q$.
\textbf{(b)} The same primes as points $\boldsymbol\theta_E(\log q)$ of the
Kronecker orbit on $\mathbb T^{2}$.  The heavy curve is the level set
$\Sigma_{E,M}(c)$, the pale band is the tube of half-width $T_E(173)$ and the
dark band that of $T_E(9973)$, so the two shadings show the same target at two
prime-dependent widths.  Zoom: $q=157$ lies inside the narrow tube, $q=163$
inside the wide one only, and $q=173$ outside both, matching their status in
panel~(a); the arrow is $\nabla\Phi_{E,M}$, whose sign against
$\operatorname{sgn}a_q(E)$ orients the intersection.  The Frobenius traces are
the exact $a_q(E)$ of the rank-one curve $y^2+y=x^3-x$; the zero field is the
two-mode truncation $\gamma=(1,\sqrt2)$ of \eqref{eq:truncated-phase-field},
chosen so that the geometry is visible on a page.  Both panels display one and
the same function, since $Z_{E,M}(u)=\Phi_{E,M}(\boldsymbol\gamma_Eu)$.}
\label{fig:crossing-definition}
\end{figure}

\subsection{A finite-dimensional Sato--Tate--phase theorem}

\paragraph{Literature.}
The unconditional inputs are the Sato--Tate theorem for non-CM curves \cite{BLGHT2011}, the symmetric-power automorphy of Newton and Thorne \cite{NewtonThorne2021I,NewtonThorne2021II}, and nonvanishing on $\Re s=1$ \cite{JacquetShalika1976}; the Weyl-criterion and partial-summation steps are as in \cite{IwaniecKowalski2004}.

For a non-CM curve write
\begin{equation}
  a_q(E)=2\sqrt q\cos\theta_q,
  \qquad \theta_q\in[0,\pi].
\end{equation}
The variable window $T_E(q)$ requires joint control of the Sato--Tate angle and the logarithmic zero phases, not phase equidistribution alone.

\begin{theorem}[Finite-dimensional amplitude--phase equidistribution]
Let $E/\Q$ be non-CM and let $\omega_1,\ldots,\omega_M$ be real numbers linearly independent over $\Q$.  Then, for every continuous function $h$ on $[0,\pi]\times\T^M$,
\begin{align}
 &\frac{1}{\log X}
 \sum_{\substack{q\le X\\q\nmid N_E}}
 \frac{\log q}{q}
 h\bigl(\theta_q,\omega_1\log q,\ldots,\omega_M\log q\bigr)
 \notag\\
 &\hspace{4em}\longrightarrow
 \int_0^\pi\int_{\T^M}
 h(\theta,\boldsymbol\phi)\,
 d\mu_{\ST}(\theta)\,d\boldsymbol\phi,
 \label{eq:joint-equidistribution}
\end{align}
where $d\boldsymbol\phi$ is normalized Haar measure.
\end{theorem}

\begin{proof}[Proof sketch]
By the Weyl criterion and the Chebyshev basis, it suffices to treat
\begin{equation}
  U_m(\cos\theta_q)\,q^{i\tau},
  \qquad \tau=\boldsymbol k\cdot\boldsymbol\omega.
\end{equation}
For $m\ge1$, Newton--Thorne automorphy and cuspidality of $\operatorname{Sym}^mE$, together with Jacquet--Shalika nonvanishing on $\Re s=1$, yield the required prime-number theorem after partial summation \cite{NewtonThorne2021II,JacquetShalika1976}.  For $m=0$ and $\boldsymbol k\ne0$, rational independence gives $\tau\ne0$; the ordinary prime-number theorem and partial summation show that the logarithmically normalized sum is $o(1)$.  The constant mode gives the total mass.
\end{proof}

\begin{remark}
The theorem is unconditional in its automorphic input, but its application to a selected list of zero ordinates retains the stated finite nonresonance assumption.  It concerns fixed-dimensional continuous test functions and fixed windows.  It does \emph{not} prove the shrinking-window local limit below, nor does it justify passage to the entire infinite zero set.
\end{remark}

\subsection{Weighted local time and Ces\`aro crossing count}

\paragraph{Literature.}
The weighted local-time normalization is the elliptic form of the statistic introduced in \cite{PlanatFalseCheb2026}; the underlying occupation-density heuristic for level sets of a flow is the coarea statement of \cref{prop:torus-local-time}, and the diagonal local limit it would require is of the type established for homogeneous flows in \cite{HillVelani1995,KleinbockMargulis1999}.

Assume that the elliptic race has a continuous limiting density $f_E$ and that a prime-sampled shrinking-window local limit holds jointly with the Frobenius amplitude.  Since the crossing window has total width
\begin{equation}
  2T_E(q)=\frac{|a_q(E)|\log q}{q},
\end{equation}
one expects
\begin{equation}
  \Pr(q\in\mathcal C_E(c))
  \sim \frac{|a_q(E)|\log q}{q}f_E(c).
  \label{eq:crossing-probability}
\end{equation}

\begin{conjecture}[Weighted elliptic crossing local time]
For a regular level $c$, define
\begin{equation}
  \mathcal L_E(X;c)
  =\frac{1}{\log X}
  \sum_{\substack{q\le X\\q\in\mathcal C_E(c)\\a_q(E)\ne0}}
  \frac{1}{|a_q(E)|}.
  \label{eq:weighted-local-time}
\end{equation}
Under ECRH and the joint prime-sampled local-limit hypothesis,
\begin{equation}
  \boxed{\mathcal L_E(X;c)\longrightarrow f_E(c).}
  \label{eq:local-time-limit}
\end{equation}
At the centred level,
\begin{equation}
  \mathcal L_E(X;\mu_E)\longrightarrow g_E(0).
\end{equation}
\end{conjecture}

For a non-CM curve,
\begin{equation}
  \mathbb E_{\ST}\!\left(\frac{|a_q(E)|}{\sqrt q}\right)
  =\frac{8}{3\pi}.
  \label{eq:ST-amplitude}
\end{equation}
A pointwise law $\#\mathcal C_E(c;X)\sim C\sqrt X$ is \emph{not} asserted: the weight $\log q/\sqrt q$ is exponentially concentrated at the endpoint in the logarithmic variable, so log-uniform phase equidistribution does not control such a limit.  The appropriate softened statistic is
\begin{equation}
  \mathcal A_E(X;c)
  =\frac{2}{\log X}
  \sum_{\substack{q\le X\\q\in\mathcal C_E(c)}}\frac{1}{\sqrt q}.
  \label{eq:Cesaro-count}
\end{equation}
Equivalently, up to a negligible boundary term, this is
\begin{equation}
  \frac{1}{\log X}
  \int_2^X\frac{\#\mathcal C_E(c;t)}{t^{3/2}}\,dt.
\end{equation}
The same heuristic gives the logarithmic Ces\`aro target
\begin{equation}
  \boxed{
  \mathcal A_E(X;c)\longrightarrow
  \frac{16}{3\pi}f_E(c).
  }
  \label{eq:Cesaro-limit}
\end{equation}
Both \eqref{eq:local-time-limit} and \eqref{eq:Cesaro-limit} remain diagonal local-limit conjectures, not consequences of \cref{eq:joint-equidistribution} alone.

\subsection{A four-stage finite-$X$ diagnostic}
\label{subsec:50700-crossing-pilot}

\paragraph{Literature.}
The ordinates are regenerated independently from modular coefficients by the method of \cite{Dokchitser2004,RubinsteinLNS2005} and checked against \cite{LMFDB}; the Frobenius traces are obtained by exact point counting as in \cite{Cremona1997}.

The conductor-$50700$ data permit a direct diagnostic after three normalization corrections: the race is the log-weighted field \eqref{eq:race}, the Fourier series is evaluated at its midpoint without an additional $H_E(q)$ translation, and the omitted prime powers that are exactly computable from the local coefficients are restored.  We use one integral representative of each of the seventeen rank-one isogeny classes and every good prime $100<q\le50000$.

The first $35$ or $36$ positive ordinates below height $20$ were independently regenerated.  The first ordinates agree with the stored values to at most $2.8\times10^{-6}$, and the completed central derivatives agree to relative error below $4.5\times10^{-16}$.  Put
\begin{equation}
  Z_{E,20}(u)
  =2\Re\sum_{0<\gamma_E\le20}
   \frac{\e^{i\gamma_Eu}}{\tfrac12+i\gamma_E}.
  \label{eq:truncated-zero-field}
\end{equation}
Let $c_E(p^m)$ denote the normalized local von Mangoldt coefficient: for a good prime, $c_E(p^m)=(\alpha_p^m+\beta_p^m)/p^{m/2}$, so $c_E(p^2)=\lambda_p^2-2$.  The exact correction for terms omitted from the good-prime first-power race is
\begin{equation}
  \mathcal P_E(q)
  =1+\frac1{\sqrt q}\left(
  \sum_{\substack{p\nmid N_E,\ m\ge2\\p^m\le q}}c_E(p^m)\log p
  +\sum_{\substack{p\mid N_E,\ m\ge1\\p^m\le q}}c_E(p^m)\log p
  \right).
  \label{eq:prime-power-correction}
\end{equation}
The $1$ cancels the stable mean $-1$ of the good-prime squares.  The prime-sampled spectral midpoint is therefore
\begin{equation}
  M_{E,20}^{*}(q)=Z_{E,20}(\log q)+\mathcal P_E(q),
  \label{eq:spectral-midpoint}
\end{equation}
not $Z_{E,20}(\log q)+H_E(q)$.

For a level offset $y=c-\mu_E$, the four stages are: the asymptotic Bessel-product density with Gaussian tail closure; finite continuous coarea local time of $Z_{E,20}$; the prime-sampled zero-field window $|M_{E,20}^{*}(q)-y|<|H_E(q)|$; and the exact arithmetic window $|B_E^{*}(q)-(\mu_E+y)|<T_E(q)$.  All discrete stages use the delayed normalization $1/\log(50000/100)$ and weight $1/|a_q(E)|$.

At the centred level $y=0$ the four stages give the ensemble statistics of \cref{tab:50700-four-stage-centre}.
\begin{table}[H]
\centering
\small
\caption{Four-stage centred diagnostic for the seventeen rank-one classes of conductor $50700$.}
\label{tab:50700-four-stage-centre}
\begin{tabular}{@{}lrrr@{}}
\toprule
Stage & Mean & Median & Total events\\
\midrule
Asymptotic Bessel density & $0.19118$ & $0.19680$ & --\\
Finite continuous zero flow & $0.24138$ & $0.15011$ & $114$\\
Prime-sampled zero field & $0.18168$ & $0.17175$ & $1293$\\
Exact log-weighted arithmetic crossings & $0.20526$ & $0.20668$ & $1271$\\
\bottomrule
\end{tabular}
\end{table}
The finite continuous mean remains unstable because the logarithmic interval has length only $6.2146\ldots$ and several curves pass near tangencies.

On the fixed grid $y=-2,-1.5,\ldots,2$, the profile errors relative to the asymptotic Bessel density are collected in \cref{tab:50700-profile-errors}.
\begin{table}[H]
\centering
\small
\caption{Ensemble profile discrepancy under the log-weighted midpoint normalization.}
\label{tab:50700-profile-errors}
\begin{tabular}{@{}lrrr@{}}
\toprule
Stage & Mean bias & RMSE & Maximum error\\
\midrule
Finite continuous zero flow & $0.02248$ & $0.04258$ & $0.05768$\\
Prime-sampled zero field & $-0.00194$ & $0.03212$ & $0.05117$\\
Exact log-weighted arithmetic crossings & $-0.00086$ & $0.02105$ & $0.03612$\\
\bottomrule
\end{tabular}
\end{table}

\Cref{fig:50700-four-stage-profile} plots the four profiles on this grid.  The three finite-$X$ stages reproduce both the scale and the gross shape of the asymptotic Bessel density, which is what the two tables quantify.  The event-level comparison below shows that this ensemble agreement does not descend to individual primes.
\begin{figure}[H]
  \centering
  \includegraphics[width=0.91\textwidth]{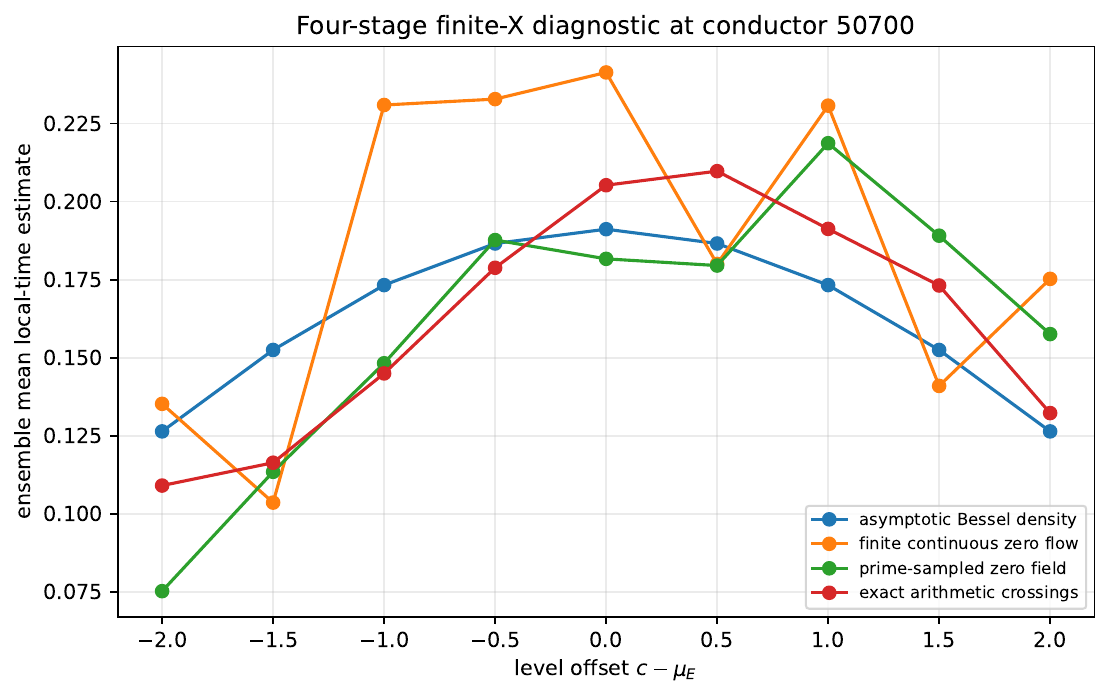}
  \caption{Four-stage finite-$X$ level profiles under the log-weighted midpoint normalization.  The close ensemble scale does not imply agreement of individual shrinking-window events.}
  \label{fig:50700-four-stage-profile}
\end{figure}

At event level, the spectral and arithmetic midpoint fields have mean correlation $0.9700$ and mean affine slope $0.9909$.  Their mean RMS discrepancy is $0.3910$, while the median Frobenius half-window is $0.02865$, a ratio of $13.6$.  The mean Jaccard overlap is $0.0633$, with recall $0.1222$ and precision $0.1380$.

These small overlaps do not diagnose a failed arithmetic transfer.  The omitted-zero variance above $T_0=20$ is
\begin{equation}
  V_E^{>20}
  \approx\frac{2}{\pi}\frac{L_E+\log20+1}{20}
  =0.2411,
  \qquad
  \sqrt{V_E^{>20}}=0.4910.
  \label{eq:tail-noise-floor}
\end{equation}
The observed mean standard deviation of the field error is $0.3878$, or $0.790$ of this predicted noise floor.  Thus the dominant uncertainty is compatible with the paper's own zero truncation.  \Cref{fig:50700-field-window} compares the residual with both the Frobenius windows and the $T_0=20$ tail scale.
\begin{figure}[H]
  \centering
  \includegraphics[width=0.82\textwidth]{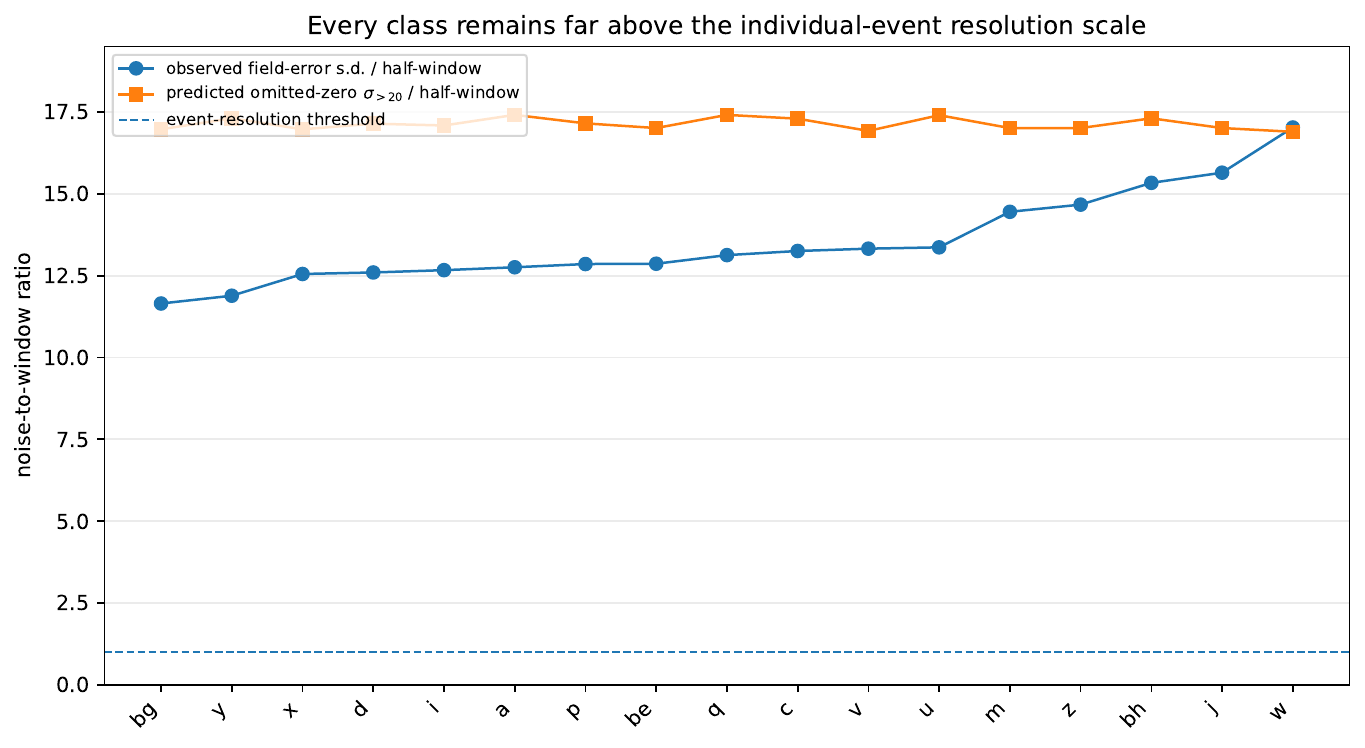}
  \caption{The classwise noise-to-window ratios remain far above the event-resolution threshold; the predicted omitted-zero contribution is shown for comparison.}
  \label{fig:50700-field-window}
\end{figure}

The correct conclusion is therefore
\begin{equation}
  \boxed{
  \text{the prime-sampled shrinking-target transfer is unobservable at }T_0=20.
  }
  \label{eq:unobservable-T20}
\end{equation}
Balancing the leading tail variance with the median half-window gives the heuristic resolution condition
\begin{equation}
  \frac{2}{\pi}\frac{L_E+\log T_0+1}{T_0}
  \lesssim(0.02865)^2,
\end{equation}
whose solution is $T_0\approx1.1\times10^4$, corresponding to roughly $4\times10^4$ positive ordinates per curve.  This is an error-budget estimate, not a sharp theorem.  Increasing the prime cutoff while retaining zeros only to height $20$ cannot remove the binding noise floor.  A decisive test requires either vastly higher zero truncation or a justified conditional/distributional treatment of the omitted tail on the crossing set.

\subsection{From a level grid to the zero product}

\paragraph{Literature.}
The characteristic function is the Bessel product of \cite{RubinsteinSarnak1994,ANS2014}, and the reconstruction of a measure on $[0,1]$ from its moments is the Hausdorff moment theorem \cite{ShohatTamarkin1943}.

The single value $g_E(0)$ is mainly a smooth concentration statistic and need not determine $J_E$.  The stronger observable is the full level profile
\begin{equation}
  y\longmapsto \mathcal L_E(X;\mu_E+y),
  \label{eq:level-profile}
\end{equation}
which is conjectured to approximate $g_E(y)$.  Its Fourier transform reconstructs $\widehat g_E$.

Under ECRH and LI, put
\begin{equation}
  b_\gamma=\frac{1}{1+4\gamma^2}\in(0,1).
\end{equation}
The Bessel product \eqref{eq:bessel-product} determines all power sums
\begin{equation}
  S_k(E)=\sum_{\gamma_E>0}b_\gamma^k,
  \qquad k\ge1,
  \label{eq:power-sums}
\end{equation}
through the Taylor coefficients of $\log\widehat g_E(t)$.  These moments determine the finite measure
\begin{equation}
  \nu_E=\sum_{\gamma_E>0}b_\gamma\,\delta_{b_\gamma}
\end{equation}
by the Hausdorff moment theorem.  Moreover,
\begin{equation}
  \boxed{
  J_E=\sum_{k\ge1}\frac{S_k(E)}{k}
  =\int_{[0,1]}
   \frac{-\log(1-x)}{x}\,d\nu_E(x).
  }
  \label{eq:J-from-density}
\end{equation}
Thus the \emph{whole} centred density determines the noncentral zero product in principle, whereas its value at a single level does not.

\begin{problem}[Prime-side spectral reconstruction]
Estimate the centred crossing local time over a grid of levels, correct its finite-$X$ bias using truncated Bessel products, reconstruct the first spectral moments $S_k(E)$, and compare the resulting approximation to $J_E$ with the direct low-zero product.  Determine whether the accuracy is sufficient for the BSD square sieve.
\end{problem}

Finite-$X$ bias must be modelled curve by curve.  Since it is governed by the same low ordinates that the statistic seeks to recover, using a common prime cutoff across a family can create a spurious first-zero signal.

\section{Discussion and outlook}
\label{sec:outlook}

The exact part of the framework is compact.  The Hadamard identity and the decomposition $J_E=V_E/8+\Delta_E$ isolate the noncentral contribution to the BSD leading coefficient; the phase observable is Morse; and the finite-dimensional coarea formula identifies its continuous level density.  The conductor-$50700$ calculation supplies two complementary diagnostics: the first-zero kernel arc and the discrete $35/36$ tail split.

Three limitations define the next mathematical steps.  First, a finite-conductor odd-orthogonal comparison requires one exact counting convention and the explicit adjacent-integer interpolation \eqref{eq:adjacent-dimension-interpolation}.  Alternative density-slope and local-density matchings should be reported only as sensitivity analyses.  Second, a non-circular BSD information interval requires an effective bound $B_E(T)$ in \cref{prop:counting-to-tail}; the density integral alone cannot certify the tail.  Third, the prime-sampled crossing problem needs control of the omitted-zero field at the scale $T_E(q)$, not merely a high correlation of coarse race values.

A conditional stochastic closure can quantify this last limitation without pretending to recover individual events.  If the computed low-zero field at $q$ is $z_{E,T_0}(q)$ and the omitted tail is modelled by $G_{E,T_0}\sim N(0,\sigma_{E,T_0}^2)$, then the conditional crossing probability is
\begin{align}
 p_{E,q}(c)=&\;\Phi\!\left(
 \frac{c+T_E(q)-z_{E,T_0}(q)}{\sigma_{E,T_0}}
 \right)\notag\\
 &-\Phi\!\left(
 \frac{c-T_E(q)-z_{E,T_0}(q)}{\sigma_{E,T_0}}
 \right).
 \label{eq:conditional-tail-closure}
\end{align}
This model can provide expected counts, uncertainty bands and power estimates, but it cannot certify whether a particular prime crosses because the omitted phases remain unknown.

Two broader directions appear especially natural.  A large certified non-CM mixed-parity family would test whether the hard-edge tail exponent distinguishes even and odd orthogonal symmetry.  On the topological side, extending the level grid from the central region toward the critical values of $\Phi_{E,M}$ would probe changes in the Morse sublevel filtration and the non-Gaussian tails of the coarea density.  Both directions are independent of the circular BSD calibration used in the present fixed-conductor audit.

\section{Conclusion}

The central zero-product statistic has the exact structure
\[
  J_E=\frac18V_E+\Delta_E,
  \qquad \Delta_E>0.
\]
For the seventeen rank-one classes of conductor $50700$, the nonlinear remainder is predominantly a first-zero quantity, and the total pairs follow the corresponding translated kernel arc.  Independent edge values separate the tail beyond height $20$ according to whether $35$ or $36$ positive ordinates have been listed, giving a direct numerical diagnostic of the zero multiset.

The random-matrix and topological conclusions are distinct but complementary.  Haar $SO(2N+1)$ produces a fixed-dimensional hard-edge exponent $3/2$, while its finite-conductor normalization requires an explicitly frozen dimension convention.  A finite zero truncation, by contrast, gives an exact Morse function on a phase torus: its critical values organize the sublevel topology, and the coarea formula gives the continuous local density.

The elliptic crossing construction transfers the midpoint and half-jump mechanism of false Chebyshev primes to Frobenius jumps.  The log-weighted conductor-$50700$ experiment shows that the zero-derived and arithmetic fields agree at a coarse level, but a height-$20$ truncation cannot resolve individual shrinking windows.  This is a quantitative resolution barrier, not evidence against the transfer.  The paper therefore supplies an exact ECRH--BSD information architecture, a finite-dimensional hard-edge theorem, a Morse--coarea model for level crossings, and a precise formulation of the remaining prime-sampled problem.

\begingroup
\small
\setlength{\parskip}{0pt}

\endgroup

\end{document}